\documentclass[a4,11pt,leqno]{amsart}
\usepackage{amsmath}
\usepackage{amsfonts}
\usepackage{mathtools}
\usepackage{amssymb}

\usepackage{epstopdf}
\usepackage{amscd}
\usepackage{color}

\swapnumbers

\newtheorem{teo}{Theorem}[section]
\newtheorem{lem}[teo]{Lemma}

\newtheorem{prop}[teo]{Proposition}
\newtheorem{es}[teo]{Example}

\newtheorem{counterex}[teo]{Counterexample}

\theoremstyle{remark}
\newtheorem{oss}[teo]{Remark}

\theoremstyle{definition}
\newtheorem{defi}[teo]{Definition}%[section]

\newcommand{\Lap}{\mathcal{L}}

\newcommand{\dist}{\mathrm{dist}}

\newcommand{\bbR}{{\mathbb{R}}}

\newcommand{\bbN}{{\mathbb{N}}}
\newcommand{\bbQ}{{\mathbb{Q}}}

\newcommand{ \Rn} {{\mathbb {R}^n}}
\newcommand{ \Rm} {{\mathbb {R}^m}}

\newcommand{\eps}{\varepsilon}
\newcommand{\vf}{\varphi}

\newcommand{\average}{{\mathchoice {\kern1ex\vcenter{\hrule height.4pt
width 6pt
depth0pt} \kern-9.7pt} {\kern1ex\vcenter{\hrule height.4pt width 4.3pt
depth0pt}
\kern-7pt} {} {} }}

\newcommand{\F}{\mathcal{F}}
\newcommand{\mA}{\mathcal{A}}

\newcommand{\AO}{\mathcal{A}_0}

\newcommand{\ci}{\mathbf{C}}

\newcommand{\Lef}{f_e}
\newcommand{\Lefh}{f_{h,e}}

\newcommand{\Czero} {{\hbox{\bf C}}_c }
\newcommand{\Co}[1]{{\bf C}^{#1}}
\renewcommand{\L}[1]{{L}^{#1}}
\newcommand{\W}[2]{{W}_X^{#1,#2} }
\newcommand{\HS}[2]{{H}_X^{#1,#2} }

\newcommand{\Lip}{{\hbox{\rm Lip}}}
\newcommand{\Om}{\Omega}
\newcommand{\ep}{\epsilon}

\newcommand{\norma}[1]{\Vert#1\Vert}
\newcommand{\Norma}[2]{{\Vert#1\Vert}_{#2}}
\newcommand{\scalare}[2]{\langle #1,#2 \rangle}
\newcommand{\vettore}[2]{(#1_1,\dots,#1_{#2})}

\makeatletter
\def\cleardoublepage{\clearpage\if@twoside \ifodd\c@page\else
\hbox{}
\thispagestyle{empty}
\newpage
\if@twocolumn\hbox{}\newpage\fi\fi\fi}
\makeatother
\title{$\Gamma$-convergence  for functionals depending on vector fields. I. Integral representation and compactness.}
\author{A. Maione}
\address{Alberto Maione: Dipartimento di Matematica\\Universit\`a di Trento\\ Via Sommarive 14\\ 38123, Povo (Trento) - Italy\\}
\email{alberto.maione@unitn.it}
\thanks{A.M is supported by MIUR, Italy, GNAMPA of INDAM and University of Trento, Italy.}
\date{\today}
\author{A. Pinamonti}
\address{Andrea Pinamonti: Dipartimento di Matematica\\Universit\`a di Trento\\ Via Sommarive 14\\ 38123, Povo (Trento) - Italy\\}
\email{andrea.pinamonti@unitn.it}
\thanks{A.P. is supported by MIUR, Italy, GNAMPA of INDAM and University of Trento, Italy.}
\author{F. Serra~Cassano}
\address{Francesco Serra Cassano: Dipartimento di Matematica\\Universit\`a di Trento\\ Via Sommarive 14\\ 38123, Povo (Trento) - Italy\\}
\email{cassano@science.unitn.it}
\thanks{F.S.C. is supported by MIUR, Italy, GNAMPA of INDAM and University of Trento, Italy.}

\begin{document}

\begin{abstract} Given a family of locally Lipschitz  vector fields $X(x)=(X_1(x),\dots,X_m(x))$ on $\Rn$, $m\leq n$, we study functionals  depending on $X$.
We prove an integral representation  for local  functionals with respect to  $X$ and a result of $\Gamma$-compactness for a class of integral functionals depending on $X$.
\end{abstract}
\maketitle
\section{Introduction}
In this paper we will deal with the $\Gamma$-convergence, with respect to $L^p(\Omega)$-topology, for integral functionals 
 $F,F_1 : \L p (\Om)\rightarrow [0,\infty]$, $1<\,\,p<\,\infty$,  defined by

\begin{equation}\label{F_p}
F(u):=
\displaystyle{\begin{cases}
\int_{\Om}f(x,Xu(x))dx&\text{ if }u\in\Co 1(\Om)%\cap\L p(\Om)
\\
\infty&\text{ otherwise}
\end{cases}
}
\end{equation}
and 
\begin{equation}\label{F1}
F_1(u):=
\displaystyle{\begin{cases}
\int_{\Om}f(x,Xu(x))dx&\text{ if }u\in W^{1,1}_{\rm loc}(\Om)%\cap\L p(\Om)
\\
\infty&\text{ otherwise}
\end{cases}
\,,
}
\end{equation}
where $X(x):=(X_1(x),\dots,X_m(x))$ is a given family of first linear differential operators, with Lipschitz coefficients on a bounded open set $\Omega\subset\Rn$, that is,
\[
X_j(x)=\,\sum_{i=1}^nc_{ji}(x)\partial_i\quad j=1,\dots,m
\]
 with $c_{ji}(x)\in Lip(\Om)$ for $j=1,\dots,m$, $i=1,\dots,n$ and where  $f:\Om\times\Rm\rightarrow [0,\infty]$  is a Borel function.  In the following, we will refer  to $X$ and $f$ as {\it $X$-gradient} and  {\it integrand function}, respectively. As usual, we will  identify each $X_j$ with the vector field $(c_{j1}(x),\dots,c_{jn}(x))\in \Lip(\Omega,\Rn)$. Moreover, we set 
\begin{equation}\label{coeffmatrvf}
C(x) = [c_{ji}(x)]_{
{i=1,\dots,n}\atop{j=1,\dots,
m}},
\end{equation}
and we will call $C(x)$  the {\it coefficient matrix of the $X$-gradient}. 
 
Throughtout the paper the class of integrand functions will tipically satisfy the following structural conditions: 
 
 \begin{itemize}
\item[($I_1$)] for every $\eta\in\Rm$, the function $f(\cdot,\eta):\,\Om\to [0,\infty]$ is Borel measurable on $\Om$;
\item[($I_2$)] for a.e. $x\in\Om$, the function $f(x, \cdot):\,\Rm\to [0,\infty)$ is convex;
\item[($I_3$)] there exists constants $c_1>\,c_0\ge\,0$ such that
\[
c_0\,|\eta|^p\le\,f(x,\eta)\le\,c_1\left(\left|\eta\right|^p+1\right)\,,
\]
for a.e. $x\in\Om$ and for each $\eta\in\Rm$.
\end{itemize}
We will denote  by $I_{m,p}(\Omega,c_0,c_1)$ the class of such integrand functions. 
%%%%%%%%%%%%%%%%%%%%%%%%%%%%%%%%%%%%%%%%%%%%%%%%%%%%%%%%%%%%
Notice that both functionals \eqref{F_p} and \eqref{F1} always admit an integral representation with respect to the Euclidean gradient. Indeed, for instance, functional \eqref{F_p} can be represented as follows
\[
F(u)=\,\int_\Om \Lef(x,Du)\,dx\text{ for each }u\in\ci^1(\Om)
\]
where $\Lef:\,\Omega\times\Rn\to [0,\infty]$  now denotes the {\it Euclidean integrand} defined as
\begin{equation}\label{Li}
\Lef(x,\xi):=\,f(x,C(x)\xi)\quad\text{ for a.e. }x\in\Om\,,\text{for each }\xi\in\Rn.
\end{equation}
Notice also that, in general, we cannot reverse this representation (see Counterexample \ref{counterex}). Moreover the representation with respect to the Euclidean gradient could yield a loss of coercivity. Indeed, for instance, let us consider as $X$-gradient the {\it Grushin }and {\it Heisenberg} vector fields in Example \ref{exvf} (ii) and (iii), respectively, and let $f(x,\eta)=\,|\eta|^2$. Then, it is easy to see that there are no positive constants $c>\,0$ such that the associated Euclidean integrand $\Lef(x,\xi)=\,f(x,C(x)\xi)=\,|C(x)\xi|^2$ satisfies
\[
\Lef(x,\xi)\ge\,c\,|\xi|^2\text{ for a.e. }x\in\Omega,\,\forall\,\xi\in\Rn\,,
\]
if  the open set $\Omega\subset\bbR^2$ contains some segment of the line $\{x_1=0\}$, for the Grushin vector field, and for each open set $\Omega\subset\bbR^3$, for the Heisenberg vector fields. Nonetheless, we will show that, by  replacing the Euclidean gradient with the $X$-gradient, we can get rid of this drawback.
%%%%%%%%%%%%%%%%%%%%%%%%%%%%%%%%%%%%%%%%%%%%%%%%%%%%%%%%%%%%%%%

Functional \eqref{F_p} was  studied in \cite{FSSC1} as far as its relaxation and in connection with the so-called Meyers-Serrin theorem for Sobolev spaces associated with the $X$-gradient, denoted $W^{1,p}_X(\Omega)$ (see Definition \ref{Definition 1.1.1} and \cite{FS}). As  a consequence, the following characterization of relaxed functionals $\bar F$ and $\bar F_1$ can be given (see \eqref{3.3}  and Theorem \ref{Theorem 3.1.1}): if $f\in I_{m,p}(\Omega,c_0,c_1)$ with $c_1\ge\,c_0>\,0$ and  $F^*:L^p(\Omega)\to [0,\infty]$ denotes the functional
\begin{equation}\label{Fstar}
F^*(u):=\displaystyle{\begin{cases}
\int_{\Om}f(x,Xu(x))dx&\text{ if }u\in W^{1,p}_{X}(\Om)
\\
\infty&\text{ otherwise}
\end{cases}
}\,,
\end{equation}
then
\begin{equation}\label{barF=barF1}
\bar F(u)=\,\bar F_1(u)=\,F^*(u)\quad\forall\,u\in L^p(\Om)\,.
\end{equation}
By \eqref{barF=barF1} and a well-known property of $\Gamma$-convergence (see \cite[Propostion 6.11]{DM}), the characterization of  $\Gamma$-limits  for functionals of type \eqref{F_p} or \eqref{F1}, associated to integrand functions in $I_{m,p}(\Omega,c_0,c_1)$, can be reduced to the one  for functionals of type  \eqref{Fstar} still associated to integrand functions in $I_{m,p}(\Omega,c_0,c_1)$. For getting such a characterization, the following structure assumption on the $X$-gradient turns out to be  a key point.

\begin{defi}\label{frkcond} We say that the family of vector fields $X(x)=\,(X_1(x),\break\dots, X_m(x))$ on an open set $\Omega\subset\Rn$ satisfies the {\it linear independence condition} (LIC)
%\red{(forse, questo nome potrebbe non essere appropriato in quanto si confonde con la proprietÃ  di Hormander)} , 
 if there exists a closed set ${\mathcal N}_X\subset\Om$ such that $|\mathcal N_X|=\,0$ and, for each $x\in\Om_X:=\,\Om\setminus\mathcal N_X$, $X_1(x),\dots, X_m(x)$ are linearly independent as vectors of $\Rn$.
\end{defi}
Let us point out that (LIC) condition embraces many relevant  families of  vector fields studied in literature (see Example \ref{exvf}). In particular neither the {\it H\"ormander condition} for $X$, that is, vector fields $X_j$'s are smooth and the rank of the Lie algebra generated by $X_1,\ldots, X_m$ equals $n$ at any point of $\Omega$, nor  the (weaker)  assumption that the $X$-gradient induces a {\it Carnot-Carath\'eodory metric} in $\Omega$ is requested. An exaustive account of these topics can be found in \cite{BLU}.

The main results of this paper are the following (see Theorems \ref{DMThm201ext} and \ref{mainthm}).
\begin{itemize}
\item Assume that the $X$-gradient satisfies (LIC) on $\Om$ and let us denote by $\mA$ the class of open sets contained in $\Omega$. Then an integral representation result, with respect to the $X$-gradient, is provided for a local functional $F:\,L^p(\Omega)\times\mA\to [0,\infty]$ satisfying suitable assumptions.

\item Assume that the $X$-gradient satisfies (LIC) on $\Om$,  and
let $F^*_h:\,L^p(\Omega)\to [0,\infty]$ ($h=1,2,\dots$) be a sequence of integral functionals  of the form \eqref{Fstar} with $f\equiv f_h$, where $(f_h)_h\subset I_{m,p}(\Om,c_0,c_1)$ for given  constants $0<\,c_0\le\,c_1$.
%\begin{equation}\label{Fstarh}
%F^*_h(u,A):=
%\displaystyle{\begin{cases}
%\int_{A}f_h(x,Xu(x))dx&\text{ if }A\in\mA,\,u\in W^{1,p}_X(A)\\
%+\infty&\text{ otherwise}
%\end{cases}
%\,,
%}
%\end{equation}
Then, up to a subsequence, $(F^*_h)_h$ $\Gamma$-converges, in   $L^p(\Om)$-topology, to a functional $F^*:\,L^p(\Omega)\to [0,\infty]$, and $F^*$ can be still represented as in \eqref{Fstar},  for a suitable  integrand function $f\in I_{m,p}(\Om,c_0,c_1)$.
\end{itemize}
We will also single out two signifiant integrand function subclasses $J_i\subset I_{m,p}(\Om,c_0,c_1)$ ($i=1,2$) for which the associated functionals in \eqref{Fstar} are still compact with respect to $\Gamma$- convergence with respect to $L^p(\Om)$-topology (see Theorem \ref{subclasscomp}). 
%\begin{equation}\label{GammalimitF}
%F(u,A):=
%\displaystyle{\begin{cases}
%\int_{A}f(x,Xu(x))dx&\text{ if }A\in\mA,u\in W^{1,p}_X(A)\\
%+\infty&\text{ otherwise}
%\end{cases}
%\,,
%}
%\end{equation}
%%%%%%%%%%%%%%%%%%%%%%%%%%%%%%%%%%%%%%%%%

The techniques for showing the integral representation Theorem \ref{DMThm201ext} rely on the analogous classical  integral representation result  for the Euclidean gradient  (see \cite[Theorem 20.1]{DM} ), together with a characterization of integral functionals depending on  the Euclidean gradient which can be also represented with respect to a given $X$-gradient (see Theorem \ref{reprvfeg}). Let us stress that    we cannot here exploit, as in the case of the Euclidean gradient, the approximation by piecewise-affine functions in classical Sobolev space $W^{1,p}(\Om)$, since it could not work  in Sobolev space $W^{1,p}_X(\Om)$ (see section 2.3).  The strategy for showing the $\Gamma$-compactness Theorem \ref{mainthm} will consists of two steps.

\noindent{\bf 1st step.} By applying classical results contained in \cite{DM}, we will prove the following result (see Theorem \ref{GammacompFheucl}): let $(f_h)_h\subset I_{m,p}(\Om,c_0,c_1)$,  let 
$(F_h)_h$ be a sequence of integral functionals on $L^p(\Omega)\times\mA$, $1<\,\,p<\,\infty$, of the form
\begin{equation}\label{Fh}
F_h(u,A):=
\displaystyle{\begin{cases}
\int_{A}\Lefh(x,Du(x))dx&\text{ if }A\in\mA,\,u\in W^{1,1}_{\rm loc}(A)\\
\infty&\text{ otherwise}
\end{cases}
\,,
}
\end{equation}
where
\begin{equation}\label{Lefh}
\Lefh(x,\xi):=\,f_h(x,C(x)\xi)\quad x\in\Om,\,\xi\in\Rn\,.
\end{equation}

Then, up to a subsequence, there exists $F:\,L^p(\Om)\times\mA\to [0,\infty]$ such that 
\begin{equation}\label{GammalimitFeuc}
F(\cdot,A)=\,\Gamma(L^p(\Om))-\lim_{h\to\infty}F_h(\cdot,A)\quad\text{for each }A\in\mA\,,
\end{equation}
and $F$ can be represented by an integral form on $W^{1,p}(A)$ by means of an {\bf Euclidean integrand function}, that is,
\begin{equation}\label{GammalimitFeucrepresW1p}
F(u,A):=\int_{A}f_e(x,Du(x))\,dx
\end{equation}
for every $A\in\mA$, for every $u\in L^p(\Om)$ such that $u|_A\in W ^{1,p}(A)$ for a suitable Borel function $f_e:\,\Omega\times\Rn\to [0,\infty]$.
\vskip5pt
\noindent{\bf 2nd step.} We will show that the class $I_{m,p}(\Om,c_0,c_1)$ satisfies the following closure property with respect to $\Gamma(L^p(\Om))$-convergence (see Theorem \ref{GammaclosureImp}): assume that $(f_h)_h\subset I_{m,p}(\Om,c_0,c_1)$ and \eqref{GammalimitFeuc} and \eqref{GammalimitFeucrepresW1p} hold, then $F$ satisfies the assumptions of the integral representation Theorem \ref{DMThm201ext}. Thus $F$ can be also represented in the integral form \eqref{Fstar}, by means of an integrand function $f\in I_{m,p}(\Om,c_0,c_1)$.  
\vskip5pt

Eventually let us point out that the $\Gamma$-convergence for functionals such as in \eqref{F_p}  have been studied in the framework of {Dirichlet forms} \cite{MR, Fu}, but for special  integrand functions $f$ and $X$-gradient satisfying the  H\"ormander condition,(see, for instance, \cite{Mo, BT} and references there in).
Other variational convergences, such  as homogenization and $H$-convergence for subelliptic PDEs   have been  also widely studied , always assuming the $X$-gradient satisfying the H\"ormander condition  (see, for instance, \cite{BMT,BPT1, BPT2, FT,FGVN,FTT,BFT,BFTT} and the references there in). In the subsequent paper \cite{MPSC} we will be concerned with relationships between $\Gamma$-convergence of  functionals depending on vector fields and convergence of their minimizers. Thus, we will refer to \cite{MPSC} for a comparison among our results with those already present in literature. 
%Moreover, in the forthcoming paper \cite{MPSC2} we plan to study homogenization in this possible degenerate setting. Thus, we will compare there our results with those contained in previous papers.
\vskip5pt

\noindent{\it Acknowledgements.} We thank A. Braides, G. Buttazzo, G. Dal Maso, A. Defranceschi and B. Franchi   for useful suggestions and discussions on this topic. %We also thank A. Braides for useful suggestions.

\section{Vector fields and Sobolev spaces depending on vector fields}

\subsection{Notation and definitions}
Through this paper 
$\Omega\subset\Rn$ is a fixed open set and $\overline{\mathbb{R}}=[-\infty,\infty]$. If $v, w\in\Rn$, we denote
by $|v|$ and $\scalare v w$ the Euclidean norm and the scalar product,
respectively. 
%If $x\in \Rn$ and $E\subset \Rn$ then $\di (x, E) =\inf\{|x-y|:y\in E\}$. 
If $\Om$ and $\Om'$ are subsets of $\Rn$ then $\Om'\Subset \Om$ means that $\Om'$ is
 compactly contained in $\Om$. Moreover, $B(x,r)$ is the open Euclidean ball of radius 
$r$ centered at $x$. Sometimes we will denote by $B^k(x,r)$ the open Euclidean ball of radius 
$r$ centered at $x\in\bbR^k$ in $\bbR^k$. If $A\subset \Rn$ then $\chi_A$ is the characteristic
function of $A$, $|A|$ is its n-dimensional Lebesgue measure $\mathcal L^n$ and by notation {\it a.e. $x\in A$}, we will simply mean $\mathcal L^n$-a.e. $x\in A$. 
%By a Radon measure $\mu$ on an open set $\Om\subset\Rn$ we will mean of a set function $\mu:\,\mathcal B(\Om)\to [0,\infty]$, where $\mathcal B(\Om)$ denote the class of Borel sets of $\Om$ satisfying $\mu(K)<\infty$ for each compact set $K\subset\Om$.
%for $k\geq0$, $\Ha k(A)$ is its $k$-dimensional Hausdorff measure(\red{serve?}).  

In the sequel we denote by $\Co k (\Omega)$ the space of $\bbR$-valued functions $k$ times
continuously differentiable and by $\Co k_c (\Omega)$  the subspace of $\Co k (\Omega)$ whose functions have support compactly contained in $\Omega$.

%by $\Lip(\Om;\Rm)$ the space of $\Rm$-valued Lipschitz
%continuous functions and we set 
%$\Co k_c(\Omega;\Rm) = \Co k (\Omega;\Rm)\cap\mathcal  E'(\Omega;\Rm)$ and 
% $\Lip_c(\Om;\Rm) = \Lip(\Om;\Rm)\cap\mathcal E'(\Omega;\Rm)$.
%Moreover, for sake of brevity, we  write $\Co k (\Omega)$ and
%$\Co k_c(\Omega)$ if $m=1$. 
%Finally $\mathcal M(\Om;\Rm)$ is the space of 
%$\Rm$-valued Radon measures.

We will use   spherically symmetric mollifiers $\rho_\ep$ defined
by $\rho_\ep(x):=\ep^{-n}\rho(\ep^{-1}|x|)$, where $\rho\in \Czero^\infty ([-1,1])$, $\rho\geq0$ 
and
$\int_0^1 \rho(t)dt=|B(0,1)|^{-1}$. 
%The latter fits the geometry of $\Rn$ meant as a {\it Carnot group} and will be denoted by $(\rho_\ep^{\bbG})_\ep$ and will be introduced later in \eqref{intrmoll}.

%We assume that $X_1,\dots,X_m$ are Lipschitz continuous vector fields on 
%$\Omega$, where $X_j=(c_{j1},\dots,c_{jn})$, $j=1,\dots,m$. We
%identify each vector field with the first order differential operator
%$\sum_i c_{ji}(x){\partial_{x_i}}$. 
%Moreover, we set 
%$X=(X_1,\dots,X_m)$ and we will refer to $X$ as {\it $X$-gradient}, while
%\begin{equation}\label{coeffmatrvf}
%C(x) = [c_{ji}(x)]_{
%{i=1,\dots,n}\atop{j=1,\dots,
%m}},
%\end{equation}
%and we will call $C(x)$  {\it coefficient matrix of $X$-gradient}. 

For any $u\in\L 1(\Om)$  define $Xu$ 
as an element of ${\mathcal D}'(\Om;\Rm)$ as follows
\[
\begin{split}
Xu(\psi):&=(X_1u(\psi_1),\dots,X_mu(\psi_m))\\
         &=-\int_\Om u\left(\sum_{i=1}^n\partial _{x_i}( c_{1,i}\,\psi_1),\dots, \sum_{i=1}^n\partial _{x_i}( c_{m,i}\,\psi_m)\right)\;dx
        % &= -\int_\Om u\,\diver( \Ct\psi)\;dx,
\end{split}
\]
$\forall\psi=\vettore {\psi} m \in\Czero^\infty(\Om;\Rm)$.
If we set 
$X^T\psi:=\,(X^T_1\psi_1,\dots, X^T_m\psi_m)$ with
\[
X_j^T\varphi:=\,\int_\Om\sum_{i=1}^n\partial _{x_i}( c_{j,i}\,\varphi)\,dx\quad\forall\,\varphi\in\ci^\infty_c(\Om),\, \forall\,j=1,\dots,m\,,
\]
 the aspect of the definition is even more familiar
$$
Xu(\psi):=-\int_\Om u X^T\psi\;dx \qquad \forall\psi \in\Czero^\infty(\Om;\Rm).
$$
%\begin{defi}\label{frkcond} We say that the family of vector fields $X(x)=\,(X_1(x),\break\dots, X_m(x))$ on an open set $\Omega\subset\Rn$ satisfies the {\it linear independence condition} (LIC)
%%\red{(forse, questo nome potrebbe non essere appropriato in quanto si confonde con la proprietÃ  di Hormander)} , 
% if there exists a closed set ${\mathcal N}_X\subset\Om$ such that $|\mathcal N_X|=\,0$ and, for each $x\in\Om_X:=\,\Om\setminus\mathcal N_X$, $X_1(x),\dots, X_m(x)$ are linearly independent as vectors of $\Rn$.
%\end{defi}
\begin{oss}{\rm If $X=\,(X_1,\dots,X_m)$ satisfies (LIC) on an open set $\Omega\subset\Rn$, then $m\le\,n$. Moreover, by the well-known extension result for Lipschitz functions, without loss of generality, we can assume that vector fields' coefficients $c_{ji}\in Lip_{\rm loc}(\Rn)$ for each $j=1,\dots,m$, $i=1,\dots,n$.
}
\end{oss}
\begin{es}[Relevant vector fields]\label{exvf}
\rm{
\item[(i)] (Euclidean gradient ) Let $X=(X_1,\dots,X_n)=\,D:=\,(\partial_{1},\dots,\partial_{n})$. In this case the coefficients matrix $C(x)$ of $X$  is a $n\times n$ matrix and 
\begin{equation}\label{euclvf}
C(x)=\,{\rm I_n}\quad\forall\,x\in\Rn\,,
\end{equation}
denoting $I_n$ the identity matrix of order $n$.

\item[(ii)] (Grushin vector fields) Let $X=(X_1,X_2)$ be the vector fields on  $\bbR^2$ defined as 
\[
X_1(x):=\,\partial_{_1},\quad X_2(x):=\,x_1\,\partial_{_2}\text{ if } x=(x_1,x_2)\in\bbR^2\,.
\]
In this case the coefficients matrix $C(x)$ of  $X$  is a $2\times 2$ matrix and 
\begin{equation}\label{grushinvf}
C(x):=\,\left[
\begin{matrix}
1&0\\
0&x_1
\end{matrix}
\right]
\end{equation}

\item[(iii)] (Heisenberg vector fields) Let $X=(X_1,X_2)$ be the vector fields on  $\bbR^3$ defined as 
\[
X_1(x):=\,\partial_{1}-\frac{x_2}{2}\partial_{3},\;X_2(x):=\,\partial_{2}+\frac{x_1}{2}\partial_{3}\text{ if } x=(x_1,x_2,x_3)\in\bbR^3\,.
\]
In this case the   coefficients matrix  $C(x)$ of  $X$  is a $2\times 3$ matrix and 
\begin{equation}\label{heisvf}
C(x):=\,\left[
\begin{matrix}
1&0&\displaystyle{-{x_2}/{2}}\\
0&1&\displaystyle{{x_1}/{2}}
\end{matrix}
\right]
\end{equation}
%\item[(iv)] \red{(Differentiable Manifolds)....}

Notice that  all three families of vector fields satisfy (LIC) respectively in $\Om=\,\Rn$, $\Omega=\,\bbR^2$ and $\Om=\,\bbR^3$ . Indeed, it suffices to take $\Omega_X=\Omega$ in $(i)$ and $(ii)$ and $\Omega_{X}=\Omega\setminus \mathcal{N}_X$ with $\mathcal N_X:=\,\{(0,x_2):\,x_2\in\bbR\}$ in $(ii)$. Moreover they are locally Lipschitz continuous in $\Omega$.

%Indeed it clear than in the case of Euclidean gradient and Heisenberg vector fields $\Om_X=\,\Om$. In the case of Grushin vector fields, by choosing $\mathcal N_X:=\,\{(0,x_2):\,x_2\in\bbR\}$ and $\Om_X:=\,\Om\setminus\mathcal N_X$, then $X_1(x)$ and $X_2(x)$ are linearly independent in $\bbR^2$ for each $x\in\Om_X$.
}
\end{es}
%Important families of vectors fields $X=\,(X_1,\dots,X_m)$ on $\Rn$ are the ones inducing a so-called {\it Carnot group structure}. Let us recall the notion of {\it Carnot group} or also  {\it stratified group} (\cite{FS,BLU} ).\red{da completare}.
%\begin{defi}\label{Cargroupgen}We say that a family of vector fields $X=\,(X_1,\dots,X_m)$ on $\Rn$ induces a Carnot group structure if..\red{da completare}
%\end{defi}
\begin{defi}\label{Definition 1.1.1} For $1\leq p\leq\infty$ we set
\[
\begin{split}
{W}_X^{1, p}(\Om)&:=\left\{ u\in L^ p(\Om):X_j u\in L^ p(\Om)\ {\hbox{\rm
for }} j=1,\dots,m\right\}\\
{W}_{X;loc}^{1, p}(\Om)&:=\left\{u: u|_{\Om'}\in W_X^{ 1,p}(\Om')\text{ for every open set }\Om'\Subset\Om\right\} 
\end{split}
\]
\end{defi}
\begin{oss} Since vector fields $X_j$ have  locally Lipschitz continuous coefficients, $\partial_{i }c_{j,i}\in L_{\rm loc}^\infty(\Rn)$ for each $j=1,\dots,m,\,i=1,\dots,n$, thus, by definition, it is immediate that, for each open bounded set $\Omega\subset\Rn$,
\begin{equation}\label{inclclassSobsp}
W^{1,p}(\Om)\subset\W1 p(\Om)\quad\forall\,p\in [1,\infty]\,,
\end{equation}
and for any $u\in W^{1,p}(\Omega)$
\begin{equation}\label{represXbyD}
Xu(x)=\,C(x)\,Du(x)\quad\text{ for a.e. }x\in\Om\,,
\end{equation}
where $W^{1,p}(\Om)$ denotes the classical Sobolev space, or, equivalently, the space $\W1 p(\Om)$ associated to $X=\,D:=\,(\partial_{x_1},\dots,\partial_{x_n})$ (see Example \ref{exvf} (i)). Moreover it is easy to see that inclusion \eqref{inclclassSobsp} can be strict and  turns  out to be continuous. As well, there is the inclusion
\begin{equation}\label{inclclassSobsploc}
W^{1,p}_{\rm loc}(\Om)\subset W_{X;{\rm loc}} ^{1,p}(\Om)\quad\forall\,p\in [1,\infty]\,,
\end{equation}

\end{oss}
The following Proposition is proved in \cite{FS}
\begin{prop}\label{psaw} $\W 1 p (\Om)$ endowed with the norm
$$
\Norma u {W_X ^{1, p} (\Om)} := \Norma u {\L p (\Om)} + \sum_{i=1}^m \Norma {X_j u} {\L 
p (\Om)}
$$
is a Banach space, reflexive if $1<p<\infty$. 
\end{prop}

% \begin{defi}\label{W1p0X} For $1\le \,p<\,\infty$, we will denote with $W^{1,p}_{X,0}(\Omega)$ the closure of $\ci^\infty_c(\Omega)$ in $(W^{1,p}_{X}(\Omega),\Norma \cdot {W_X ^{1, p} (\Om)} )$.
% \end{defi}

\begin{oss}\label{extpropoW1pXloc}The following properties hold for functions in $W_{X;{\rm loc}}^{1,p}(\Om) $: 
\begin{itemize}
	\item[(i)] let $u\in L^p(\Om)$ and assume there exists an open set $A\subset\Om$ such that $u|_A\in W_{X;{\rm loc}} ^{1,p}(A)$. Then, for every open set $A'\Subset A$, there exists 
\begin{equation}\label{extpropw}
w\in\W1p(\Om) \text{ such that }u|_{A'}=\,w|_{A'}\,.
\end{equation}
 Indeed, there exists a cut-off function $\vf\in\Co1_c(A)$ such that $\vf\equiv 1$ in $A'$. If 
\[
w(x):=\,u(x)\,\vf(x) \text{ if }x\in\Om\,,
\]
then it is easy to see that $w$ satisfies \eqref{extpropw}.
\item[(ii)] Let $\{A_1,\ldots, A_N\}$ be a finite family of open subsets of $\Omega$ and let $u\in L^p(\Omega)$. If $u_{|A_i}\in W^{1,p}_X(A_i)$ for all $i=1,\ldots, N$ then $u \in  W_{X}^{1,p}\left(\bigcup_{i=1}^{N} A_i\right)$.
Consider a partition of unity subordinate to the covering $\{A_1,\ldots, A_N\}$, i.e., nonnegative functions $\{\eta_1,\ldots, \eta_N\} \subset C^{\infty}_c\left(\bigcup_{i=1}^{N} A_i\right)$ such that each $\eta_j$ has support in some $A_i$ and $\quad \sum_{j=1}^N \eta_j(x)=1$ for all $x\in \bigcup_{i=1}^{N} A_i$. Set $u_j=u\eta_j$. Since the support of $\eta_j$ is contained in some $A_i$, it is clear that  $u_j\in W_{X}^{1,p}\left(\bigcup_{i=1}^N A_i\right)$. The conclusion follows observing that $u=\sum_{j=1}^N u_j$.

\item[(iii)] Let $A\subset \Omega$ be an open subset and let $u\in L^p(A)$ be such that there exists $M>0$, $\|u\|_{W_{X}^{1,p}(A')}\leq M$ for any $A'\Subset A$, then $u\in W_{X}^{1,p}(A)$. It is easy to see that $u$ admits the weak gradient $Xu$. Consider a sequence of open subsets of $A$, $\{A_i\}_{i\in \mathbb{N}}$ with $A_i\Subset A_{i+1}$ and $A\subseteq \bigcup_{i=1}^{\infty} A_i$
\[
\int_{A} |X u|^p\, dx\leq \int_{\bigcup_{i=1}^{\infty} A_i} |Xu|^p\, dx=\lim_{i\to \infty}\int_{ A_i} |X u|^p\, dx\leq M
\]
and the conclusion follows.
\item[(iv)] Let $A\subset\Omega$ be an open subset and $u\in W_{X}^{1,p}(A)$, then $u_{|B}\in W_{X}^{1,p}(B)$ for any open set $B\subseteq A$. The thesis follows easily observing that $C^{\infty}_c(B)\subseteq C^{\infty}_c(A)$.
\end{itemize}
 
\end{oss}

\subsection{Approximation by regular functions}
Let us recall in this section some results of approximation by regular functions in these anisotropic Sobolev spaces. In particular the analogous of the celebrated Meyers-Serrin theorem, proved, independently, in \cite{FSSC1} and \cite{GN}.
Analogous results (under some additional assumptions) in the weighted cases are proved in \cite{FSSC2}, see also \cite{APS} for a generalization to metric measure spaces.

Here and in the sequel, if $u:\Om\to\bar\bbR$, we will denote by $\bar u:\Rn\to\bar\bbR$ its extension to the whole $\Rn$ being $0$ outside of $\Omega$.

\begin{prop}\label{Proposition 1.2.2}
Assume $u\in\W 1 p (\Om) $ for $1\leq p<\infty$.  Then if $\Om'\Subset\Om $
$$
\lim_{\ep\to 0}\Norma {\bar u\ast \rho_\ep-u} {\W 1 p (\Om')}=0,
$$
where $\rho_{\ep}(x)=\ep^{-n}\rho(\ep^{-1}|x|)$ is a
 mollifier supported in $B(0,\ep)$.
\end{prop}

\begin{defi}\label{Definition 1.1.3} For $1\leq p\leq\infty$ we set
$$
\HS 1 p (\Om):= {\hbox{\rm closure of }} \Co 1 (\Om)\cap {\W 1 p (\Om)} \;
{\hbox{\rm in  }} {\W 1 p (\Om)}
$$ 
\end{defi}
As for the usual Sobolev spaces $ \HS 1 p (\Om)\subset {\W 1 p (\Om)}$. The classical  result \lq$H=W$\rq of Meyers and Serrin (\cite{MS}) still holds for  these anisotropic Sobolev spaces.

\begin{teo}\label{Theorem 1.2.3} Let $\Om$ be an open subset of $\Rn$ and $1\leq p <\infty$. Then 
$$ \HS 1 p (\Om)= {\W 1 p (\Om)}.$$
\end{teo}
The proofs of Proposition \ref{Proposition 1.2.2} and Theorem \ref{Theorem 1.2.3} can be found in \cite{FSSC1} and \cite{GN}.

 Let us collect below some well-known properties about approximation by convolution and convex functions.

%An intrinsic approximation by convolution can be carried out within the Carnot groups (see \cite{FS} and \cite{CM}). In particular the following result holds.
\begin{prop}\label{convolconv}
\begin{itemize}
\item[(i)] Let $(u_h)_h$ and $u$ be in $L^p_{loc}(\Rn)$ and let $\Omega\subset\Rn$ be a bounded open set such that
\[
%(u_h|_\Om)_h\text{ and }\,u|_\Om\text{ are in } W^{1,p}_X(\Om)\text{ and } 
u_h\rightarrow u\text{ in }L^1_{loc}(\Om)\text{ as }h\to\infty\,.
\]
Then, for each open set $\Om'\Subset\Om$, for given $0<\,\ep<\,\dist(\Om',\Rn\setminus\Om)$,
\begin{equation}\label{convunifrhoepuh}
\rho_\ep\ast u_h\rightarrow \rho_\ep\ast u\text{ uniformly on }\Omega'\text{, as }h\to\infty\,.
\end{equation}
%and, 
%for each $j=1,\dots,m$,
%\begin{equation}\label{convunifrhoepXuh}
%\rho_\ep\ast X_ju_h\rightarrow \rho_\ep\ast X_j u\text{ uniformly on }\Omega_\ep^R\text{, as }h\to\infty\,.
%\end{equation}
\item[(ii)] Let $f:\,\Rm\to [0,\infty)$ be a convex function and let $w\in L^1_{\rm loc}(\Rn,\Rm)$. Then, for each bounded open sets $\Om'$ and $\Om$ with $\Om'\Subset\Om$, for each $0<\,\ep<\,\dist(\Om',\Rn\setminus\Om)$,
\[
\int_{\Om'}f(\rho_\ep\ast w)\,dx\le\,\int_{\Om}f( w)\,dx.
\]
\end{itemize}
\end{prop}
\begin{proof}(i) See, for instance, \cite[Proof of Theorem 23.1] {DM}

\noindent(ii) See, for instance, \cite[(23.5)]{DM}.
\end{proof}

\subsection{Approximation by piecewise affine functions}
It is well known (see, for instance, \cite[Chap. X, Proposition 2.9]{ET}) that the class of piecewise affine functions is dense in the classical Sobolev space $W^{1,p}(\Om)$, provided that $\Omega$ is a bounded open set with Lipschitz boundary. This result is crucial in the proof of the classical integral representation theorem with respect to  the Euclidean gradient (see, for instance, \cite[Theorem 20.1]{DM}).
The aim of this section is to prove that no results of this kind are available for a general family $X=(X_1,\ldots, X_m)$ in $\mathbb{R}^n$, by extending, in a natural way,  the notion  to be affine with respect to the $X$-gradient.
We say that $u\in\ci^\infty(\Rn)$ is {\it $X-$affine} if there exists $c\in \mathbb{R}^n$ such that
$Xu(x)=c$ for all $x\in \mathbb{R}^n$. Let $\Omega\subset\mathbb{R}^n$ be open. We say that $u:\Omega\to \mathbb{R}$ is $X-$affine if it is the restriction to $\Omega$ of a $X-$affine function over $\mathbb{R}^n$. Moreover, we say that $u:\mathbb{R}^n\to \mathbb{R}$ is $X-$piecewise affine if it is continuous and there is a partition of $\mathbb{R}^n$ into a negligible set and a finite number of open sets on which $u$ is $X-$affine.
We prove that for Grushin and Heisenberg vector fields the approximation of functions in $W^{1,p}_X(\Omega)$ using $X-$piecewise affine functions does not hold.\\
 It is easy to see that, if $X=(X_1,X_2)$ is the Heisenberg vector field on $\bbR^3$  (see  Example \ref{exvf} (iii)), then a function $u\in\ci^\infty(\bbR^3)$ is $X-$affine if and only if  
 \begin{equation}\label{intrlinfuncheis}
 u(x)=c_1 x_1+ c_2 x_2+c_3\text{ for each }x=(x_1,x_2,x_3)\in\bbR^3\, ,
 \end{equation}
 for suitable constants $c_i\in\bbR$ $i=1,2,3$. Indeed, it is trivial that a function $u$ in \eqref{intrlinfuncheis} is $X$- affine. Conversely, if $X_1u=c_1$ and $X_2 u=c_2$ on $\bbR^3$, for some $u\in\ci^\infty(\bbR^3)$, then the commutator $[X_1,X_2]u:=\,(X_1X_2-X_2X_1)u=\partial_3 u=0$ on $\bbR^3$, which gives $u(x)=c_1 x_1+ c_2 x_2+c_3$ for each $x=(x_1,x_2,x_3)\in\bbR^3$, for some $c_3\in\mathbb{R}$.
 
%In the Heisenberg group there is a natural notion of affine functions (\cite{Ma}), namely $u:\mathbb{H}^1\to \mathbb{R}$ is intrinsic affine if there exist $(v_1,v_2)\in \mathbb{R}^2$ and $c\in\mathbb{R}$ such that $u(x,y,t)=v_1 x+ v_2 y+c$. Clearly, any such function is $X-$affine. 
%
 Let $u(x)=x_3$, then $u\in W^{1,p}_X(\Omega)$ whenever $|\Omega|<\infty$. Since any $X-$piecewise affine function does not depend on $x_3$, there cannot be any sequence of $X-$piecewise affine functions $(u_h)_h$ such that $u_h(x_1,x_2,x_3)\to u(x_1,x_2,x_3)$ for a.e. $(x_1,x_2,x_3)\in \Omega$.

Let $X=(X_1,X_2)$ be the Grushin vector fields on $\bbR^2$ (see  Example \ref{exvf} (ii)). Let $u\in\ci^\infty(\bbR^2)$ be such that $X_1u=c_1$  and $X_2u=c_2$ on $\bbR^2$. Then it is easy to prove, arguing as before, that $u(x)=c_1 x_1+c_3$ for each $x=(x_1,x_2)\in\bbR^2$, for some $c_3\in \mathbb{R}$. The conclusion follows as in the previous case taking $u(x_1,x_2)=x_2$, which belongs to $W^{1,p}_X(\Omega)$ for any $p\geq 1$ and any bounded  open set $\Omega\subset\mathbb{R}^2$.

\section{Relaxation and characterization of integral functionals depending on vector fields}

%We are going to deal with functionals $F,F_1 : \L p (\Om)\rightarrow [0,\infty]$ defined by
%
%\begin{equation}\label{F_p}
%F(u):=
%\displaystyle{\begin{cases}
%\int_{\Om}f(x,Xu(x))dx&\text{ if }u\in\Co 1(\Om)%\cap\L p(\Om)
%\\
%+\infty&\text{ otherwise}
%\end{cases}
%}
%\end{equation}
%\begin{equation}\label{F1}
%F_1(u):=
%\displaystyle{\begin{cases}
%\int_{\Om}f(x,Xu(x))dx&\text{ if }u\in W^{1,1}_{\rm loc}(\Om)%\cap\L p(\Om)
%\\
%+\infty&\text{ otherwise}
%\end{cases}
%}
%\end{equation}
%where $p\ge 1$ and  $f:\Om\times\Rm\rightarrow [0,\infty)$  is a Carath\'eodory function with
%?????
%\begin{equation}\label{3.1}
%f(x,\cdot)\text{ is a convex function on }\Rm \text{ a.e. 
%}x\in\Om\, ;
%\end{equation}
%and verifying one or both of the following conditions:
%there exist a positive constant $b$ and a non negative function $a(x)\in\L1(\Om)$ for which
%\begin{equation}\label{3.2}
%0\le f(x,\eta)\le\,  a(x)+\,b\,|\eta|^p \text{ a.e. 
%}x\in\Om,\forall\,\eta\in\Rm\,;
%\end{equation}
%there exists a positive constant $c_0$ such that
%\begin{equation}\label{3.2.1}
%c_0|\eta|^p\le f(x,\eta)\text{ a.e. 
%}x\in\Om,\forall\,\eta\in\Rm\,.
%\end{equation}
In the study of the $\Gamma-$convergence it will be helpful to consider $F$ and $F_1$  as {\it local functionals}. Namely, according to \cite[Chap. 15]{DM}, we will consider the functionals $F,\,F_1:\,L^p(\Om)\times\mA\to [0,\infty]$
\begin{equation}\label{incrFp}
F(u,A):=
\begin{cases}
 \int_{A}f(x,Xu(x))dx&\text{ if }A\in\mA \text{ and }u\in\Co 1(A)\cap\L p(A)
 \\
\infty&\text{ otherwise}
\end{cases}
\end{equation}
\begin{equation}\label{incrF1}
F_1(u,A):=
\begin{cases}
 \int_{A}f(x,Xu(x))dx&\text{ if }A\in\mA \text{ and }u\in W^{1,1}_{\rm loc}(A)\cap\L p(A)
 \\
\infty&\text{ otherwise}
\end{cases}
\end{equation}
%where $\mA$  denotes the class of all open subsets of $\Omega$. 
For future use, we denote by $\AO$ the class of all open sets compactly 
contained in $\Om$.
%%%%%%%%%%%%%%%%%%%%%%%%%%%%%%%%%%%%%%%%%%%%%%%%%%%%%%%%%%%%%%%%
\subsection{Characterization of the relaxed functional and its finiteness domain} We are going to characterize the {\it relaxed functionals of $F$} in \eqref{F_p}   and $F_1$ in \eqref{F1} with respect to the topology of $L^p(\Om)$. Let us recall that the  relaxed functional of a given functional $G:\,L^p(\Om)\to [0,\infty]$ is defined as follows (see, for instance, \cite{B}):
\begin{equation}\label{3.3}
\bar G(u):=\inf\left\{\liminf_{h\to \infty}G(u_h): (u_h)_h\subset\L p(\Om)
, u_h\to u\text{ in }\L p(\Om)\right\}\,.
\end{equation}
Then it is well known (see, for instance, \cite{B}) that $\bar G$ is the
greatest $\L p(\Om)$-lower semicontinuous functional smaller or equal to
$G$.

%and coinciding with $F$ on $\Co 1(\Om)$. Analogously, let us define $\bar F_1$ the relaxed functional of the functional $F_1$ in \eqref{F1}.
%
The relaxed functionals $\bar F$ and $\bar F_1$ can be characterized as follows:

\begin{teo}\label{Theorem 3.1.1} Let $p>1$ and let $\Om$ be an open subset of $\Rn$;
let $f:\Om\times\Rm\rightarrow [0,\infty)$ be an integrand function in $I_{m,p}(\Omega, c_1,c_0)$ with $c_1\ge\,c_0>\,0$. Then
\item{$(i)$} {\rm dom}$\,\bar F:=\left\{u\in\L p(\Om): \bar F(u)<\infty\right\}= 
\W 1 p (\Om)$ ;
\item{$(ii)$} $\bar F(u)=\int_{\Om}f(x,Xu(x))\,dx$ for every $u\in\W 1 p(\Om)$;
\item{$(iii)$} $\bar F(u)=\,\bar F_1(u)$ for each $u\in L^p(\Om)$.
\end{teo}
\begin{proof} 
Claims  (i) and (ii) are proved in  \cite[Theorem 3.3.1]{FSSC1}. Let us prove (iii). Let $u\in L^p(\Omega)$ and $(u_h)_{h}\subset \ci^1(\Omega)\cap L^p(\Omega)$ with $u_h\to u$ in $L^p(\Omega)$. Since, in particular, $(u_h)_{h}\subset W^{1,1}_{loc}(\Omega)\cap L^p(\Omega)$ we get
\begin{align}
\bar{F}_1(u)\leq \liminf_{h\to\infty} \bar{F}_1(u_h)=\liminf_{h\to\infty}\int_{\Omega} f(x,Xu_h)\, dx=\liminf_{h\to\infty} F(u_h)
\end{align}
which implies 
\begin{align}\label{sxc}
\bar{F}_1(u)\leq \bar{F}(u).
\end{align}
Let  $F^*:\, L^p(\Omega)\to [0,\infty]$ denote the functional in \eqref{Fstar}.
By \cite[Theorem 2.3.1]{B}, $F^*$ is $L^p(\Om)$-lower semicontinuous. 
Let $u\in {\rm dom}(F_1):=\{v\in W^{1,1}_{loc}(\Omega)\cap L^p(\Omega)\ |\ F_1(u)<\infty\}$, then, by ($I_3$), we have
\[
c_0\int_{\Omega}|Xu|^p\, dx\leq  F_1(u)<\infty,
\]
thus  $u\in W^{1,p}_{X}(\Omega)$ and ${\rm dom}\,  F_1\subset W^{1,p}_{X}(\Omega)$. This implies $F^*\leq F_1$ on $L^p(\Omega)$ and consequently $F^*\leq \bar{F}_1$ on $L^p(\Omega)$. Using \eqref{sxc} and (ii) we conclude
\[
F^*\leq \bar{F}_1\leq \bar{F}=F^*\quad \mbox{on}\quad  L^p(\Omega)
\]
which completes the proof.

%Conversely, if $u\in W^{1,p}_X(\Omega)$, there exists $(u_n)_{n\in\mathbb{N}}\subset C^1(\Omega)\cap W^{1,p}_X(\Omega)$ such that $u_n\to u$ in $W^{1,p}_X(\Omega)$. Since, in particular, $(u_n)_{n\in\mathbb{N}}\subset W^{1,1}_{loc}(\Omega)\cap W^{1,p}_X(\Omega)$ and by the semicontinuity of $\bar{F}_1$ and \eqref{3.2} we get
%\begin{align*}
%\bar{F}_1(u)\leq \liminf_{n\to\infty} \int_{\Omega} f(x,Xu_n)\, dx<\infty
%\end{align*}
%and ${\rm dom}\,\bar F_1= W^{1,p}_X(\Omega)$.
\end{proof}
When $p=1$ the domain of relaxed functional $\bar F$ gives rise to the space of functions of bounded variation  function associated to $X$, $BV_X(\Omega)$ (see \cite[Theorem 3.2.3]{FSSC1}).

%%%%%%%%%%%%%%%%%%%%%%%%%%%%%%%%%%%%%%%%%%%%%%%%%%%%%%%%%
\subsection{A characterization of functionals depending on vector fields} We are going to study  when a local functional $F:\,\Co1(\Om)\times\mA\to [0,\infty]$
 can be equivalently represented both with respect to  a family of vector fields $X$ and the Euclidean gradient $D$.
 
 We already stressed that the functional $F$ in \eqref{F_p} can be always represented with respect to  the Euclidean gradient on $\Co1(\Om)$ by means of the Euclidean integrand \eqref{Li}. 
% Indeed, let us define what we will call {\it Euclidean integrand} for the functional $F$ in \eqref{F_p}, that is the function
%\begin{equation}
%\Lef (x,\xi):=\,f\left(x,C(x)\,\xi\right)\quad x\in\Omega,\,\xi\in\Rn\,.
%\end{equation}

Then, it is clear that, for each $A\in\mA$ and  $u\in\Co1(A)$,
\begin{equation}\label{agrfunct}
\begin{split}
F(u,A)&=\int_A f(x,Xu)\,dx=\,\int_A f(x,C(x)Du)\,dx\\
&=\,\int_A \Lef(x,Du)\,dx\,.
\end{split}
\end{equation}
Viceversa, we are going to study when, given a $X$-gradient and  a functional $F:\,\Co1(\Om)\times\mA\to [0,\infty]$
\begin{equation}\label{mfonC1}
F(u,A)=\int_A \Lef(x,Du)\,dx\quad u\in\Co1(A)\,,
\end{equation}
there exist a function $f:\Om\times\Rm\rightarrow [0,\infty]$ such that
\[
F(u,A)=\int_A f(x,Xu)\,dx.
\]
Let us begin with some preliminaries of linear algebra.

In the following, we identify the space of real matrices of order $m\times n$ with $\bbR^{mn}$ or $\mathcal L(\Rm,\Rn) $, where $\mathcal L(\Rm,\Rn)$ denotes the class of linear maps from $\Rm$ to $\Rn$ endowed with its operator norm. Given a matrix $A=[a_{ij}]$ of order $m\times n$ its {\it operator norm} is defined as
\[
\|A\|:=\,\sup_{|z|=1}|Az|
\]
and its {\it Hilbert-Schmidt norm} as
\[
\|A\|_{\bbR^{mn}}:=\,\sqrt{\sum_{i,j} a^2_{ij}}
\]
(see \cite[Chap. 7]{La}). Since the norms are equivalent, we can also identify the spaces
\begin{equation}
\Co0(\Om_X,\bbR^{mn})\equiv\Co0(\Om_X,\mathcal L(\Rm,\Rn))\,,
\end{equation}
where we recall that $\Om_X=\Omega\setminus \mathcal{N}_X$.
For each $x\in\Om$, let $L_x:\,\Rn\to\Rm$ be the linear map
\begin{equation}\label{Lx}
L_x(v):=\,C(x)v\text{ if }v\in\Rn\,
\end{equation}
where $C(x)$ denotes the matrix in \eqref{coeffmatrvf}.
Let $N_x$ and $V_x$ respectively denote the subspaces of  $\Rn$ defined as
%\begin{equation}\label{N&Vx}
%N_x:={\rm ker}(L_x),\quad V_x:=\,L_x^{-1}(\Rm)\,.
%\end{equation}
\begin{equation}\label{N&Vx}
N_x:={\rm ker}(L_x),\quad V_x:=\,\{C(x)^T z\ :\ z\in\mathbb{R}^m\}.
\end{equation}
It is well-known that $N_x$ and $V_x$ are orthogonal complements in $\Rn$, that is
\begin{equation}
\Rn=N_x\oplus V_x\,.
\end{equation}
Moreover, for each $x\in\Om$ and $\xi\in\Rn$, let us define $\xi_{N_x}\in N_x$ and $\xi_{V_x}\in V_x$ as the unique vectors of $\Rn$ such that
\begin{equation}\label{splitting}
\xi=\,\xi_{N_x}+\,\xi_{V_x}
\end{equation}
 and let $\Pi_x:\,\Rn\to V_x\subset\Rn$ be the projection
 \begin{equation}\label{Pix}
 \Pi_x(\xi):=\,\xi_{V_x}.
 \end{equation}
\begin{prop}\label{decompRnbyVx} Assume that the family $X$ of vector fields satisfies (LIC) on $\Omega$. Let $C(x)$ be the matrix in \eqref{coeffmatrvf} and $L_x$  be the map in \eqref{Lx}. Then $L_x:\,V_x\to\Rm$ is invertible and the map $L^{-1}:\,\Om_X\to \mathcal L(\Rm,\Rn)$ defined as
\begin{equation}\label{L-1}
L^{-1}(x):=\,L^{-1}_x\ \text{if }x\in\Om_X
\end{equation}
belongs to $\Co0(\Om_X,\mathcal L(\Rm,\Rn))$.
\end{prop}
Before giving the proof of Proposition \ref{decompRnbyVx}, let us prove a preliminary technical lemma.
\begin{lem}\label{helplemma}Under the same assumptions of Proposition \ref{decompRnbyVx}, 
\begin{itemize}
\item[(i)] $\dim V_x=\,m$ for each $x\in\Omega_X$ and $L_x(V_x)={\rm range}(L_x)=\,\Rm$
where ${\rm range}(L_x)$ denotes the range of $L_x$, that is, ${\rm range}(L_x):=\,\{L_x(v):\,v\in\Rn\}$. In particular $L_x:\,V_x\to\Rm$ is an isomorphism.

\item [(ii)] Let
\begin{equation}\label{matrixB}
B(x):=\,C(x)\,C^T(x)\quad x\in\Om\,.
\end{equation}
Then, for each $x\in\Om_X$, $B(x)$ is a symmetric invertible matrix of order $m$. Moreover the map $B^{-1}:\,\Om_X\to\mathcal L(\Rm,\Rm)$, defined as
\begin{equation}\label{L1}
B^{-1}(x)(z):=\,B(x)^{-1}z\quad\text{if }z\in\Rm\,,
\end{equation}
 is continuous.
%\item [(iii)] For each $x\in\Om_X$,
%\[
%V_x=\left\{C^T(x)z:\,z\in\Rm\right\}\,.
%\]
\item [(iii)] For each $x\in\Om_X$, the projection $\Pi_x$ in \eqref{Pix} can be represented as
\[
\Pi_x(\xi)=\,\xi_{V_x}=\,C(x)^TB(x)^{-1}C(x)\,\xi,\quad\forall\,\xi\in\Rn\,.
\]
 If $m=n$, then, $\Pi_x=\,{\rm Id}_n:\,\Rn\to\Rn$, the identity map in $\Rn$.
\end{itemize}
\end{lem}
\begin{oss} Using the definition of $V_x$ it is easy to see that
\[
V_x=\,{\rm span}_{\bbR} \left\{X_1(x),\dots,X_m(x)\right\}\,,
\]
i.e., the so-called {\it horizontal bundle}, denoted also by $H_x$.
\end{oss}

\begin{proof}(i) The claim is a well-known result of basic linear algebra.

(ii)  It is straightforward that $B(x)$ a symmetric matrix of order $m$ for each $x\in\Om$. We have only to show that it is invertible for each $x\in\Om_X$ or, equivalently, that
\begin{equation}\label{helplemma1}
\text{if $B(x)z=\,0$ for some $z\in\Rm$, then $z=\,0$.}
\end{equation}
Let $z^T$ denotes the transpose of a column vector $z\in\Rm$. If $B(x)z=\,0$, then
\begin{equation}\label{helplemma2}
\begin{split}
0=\,z^TB(x)z&=z^TC(x)C^T(x)z\\
&=\,\left|C^T(x)z\right|^2_{\Rm}\quad\iff\quad C^T(x)z=\,0\,.
\end{split}
\end{equation}

By (LIC), since 
\[
{\rm rank }\,C( x)=\,{\rm rank }\,C^T( x)=\,m\quad\forall x\in\Om_X\,,
\]
from \eqref{helplemma2} we get that $z=0$ and \eqref{helplemma1} follows. Let us now prove that the map \eqref{L1} is continuous.
Let us recall that, given  a matrix $A\in\Co0(\Om_X,\bbR^{m^2})$, by the definition of determinant (see, for instance, \cite[Chap.3,Theorem 6]{La}), the determinant map 
\[
\det A:\,\Om_X\to\bbR,\quad(\det A)(x):=\,\det(A(x))
\]
is continuous. Moreover 
\[
\text{$A(x)$ is invertible}\iff \det A(x)\neq\,0\,.
\]
By Cramer's rule (see, for instance, \cite[Chap.3,Theorem 7]{La}), if $B(x)^{-1}=[b_{ij}^*(x)]$, then
\[
b^*_{ij}(x)=\,(-1)^{i+j}\frac{\det B_{ij}(x)}{\det B(x)}\quad x\in\Om_X,\,i,j=1,\dots,m\,,
\]
where $B_{ij}$ is the $(m-1)\times (m-1)$ matrix obtained by striking out the $i$th row and $j$th column of $B$, i.e., the $(ij)$th minor of $B$. This implies that $B^{-1}\in\Co0(\Om_X,\bbR^{m^2})\equiv\Co0(\Om_X,\mathcal L(\Rm,\Rm))$.

%(iii) Let us fix $x\in\Om_X$ and 
% prove the inclusion

%\begin{equation}\label{inclVx}
%V_x\supset\left\{C^T(x)z:\,z\in\Rm\right\}\,.
%\end{equation}

%Let $v=C^T(x)z$ for some $z\in\Rm$. If $v\in{\rm ker}(L_x)$, then
%\[
%0=\,L_x(v)=\,C(x)v=C(x)C^T(x)z=B(x)z\,.
%\]
%Thus, from previous claim (i), we get $z=0$. This implies that $L(v)\neq\,0$ if $z\neq\,0$,  and then, by definition, \eqref{inclVx} holds. Let us now show the reverse inclusion of \eqref{inclVx}. Let $v\in V_x\setminus\{0\}$ and let $w=\,L_x(v)\in\Rm\setminus\{0\}$. By claim (i) there exists $z\in\Rm$ such that
%\[
%w=\,B(x)z =\,C(x)C^T(x)z= L_x\left(C^T(x)z\right)\,.
%\]
%Since $L_x:\,V_x\to\Rm$ is an isomorphism, it follows that $v= C^T(x)z$ and we are done.

(iii) %By the previous claim (iii), it follows that
We have
\begin{equation}\label{represPix1}
\Pi_x(\xi)=\,\xi_{V_x}=\,C(x)^T w\,
\end{equation}
for a suitable (unique) $w=w(x,\xi)\in\Rm$ depending on $x$ and $\xi$. On the other hand, by \eqref{represPix1},
\begin{equation}\label{represPix2}
\begin{split}
C(x)\xi&=\,L_x(\xi)=\,L_x(\xi_{N_x})+\,L_x(\xi_{V_x})\\
&=\,C(x)\xi_{V_x}=\,C(x)C(x)^Tw=\,B(x)w\,.
\end{split}
\end{equation}
Since $B(x)$ is invertible, by \eqref{represPix2}, we get the desired conclusion.
\end{proof}
\begin{proof}[Proof of Proposition \ref{decompRnbyVx}] The fact that the map $L_x:\,V_x\to\Rm$ is invertible follows from Lemma \ref{helplemma} (i). Let us now prove that
\begin{equation}\label{L-1repr}
L_x^{-1}(z)=\,C^T(x)B(x)^{-1}z\quad\forall\,z\in\Rm\,,
\end{equation}
where $B(x)$ is the matrix in \eqref{matrixB}. Let us fix $z\in\Rm$ and let $v=\,L_x^{-1}(z)\in V_x$. By Lemma \ref{helplemma} (iii), there exists $w\in\Rm$ such that $v=\,C^T(x)w$. Thus
\[
z=\,L_x(v)=\,C(x)C^T(x)w=\,B(x)w\,.
\]
By Lemma \ref{helplemma} (ii), it holds $w=\,B(x)^{-1}z$. Therefore we get
\begin{equation}\label{represL-1}
L_x^{-1}(z)=\,v=\,C^T(x)B(x)^{-1}z
\end{equation}
and \eqref{L-1repr} follows.  Let us define
\[
A(x):=\,C^T(x)B(x)^{-1}\quad\text{if }x\in\Om_X\,.
\]
Then, from Lemma \ref{helplemma} (ii), $A\in \Co0(\Om_X,\bbR^{mn})\equiv\Co0(\Om_X,\mathcal L(\Rm,\Rn))$. Thus, by \eqref{represL-1}, we get the desired conclusion.
%Let us define the maps $L^{(1)}:\Om\to\mathcal L(\Rm,\Rm)$ and $L^{(2)}:\Om\to\mathcal L(\Rm,\Rn)$ as
%\[
%L^{(1)}(x)(z):=\,B(x)^{-1}z,\quad L^{(2)}(x)(z):=\,C^T(x)z\quad\text{ if }x\in\Om_X,\,z\in\Rm\,.
%\]
%Then by Lemma \ref{helplemma} (ii), 
%\begin{equation}
%\text{ $L^{(i)}$ are continuous for $i=1,2$} 
%\end{equation}
%and
%\begin{equation}
%L^{-1}(x)=\,L^{(2)}(x)\circ L^{(1)}(x)\quad\forall\,x\in\Om_X\,.
%\end{equation}
\end{proof}
\begin{teo}\label{reprvfeg} Let $\Om\subset\Rn$ be an open set and assume that $X$ satisfies (LIC) on $\Omega$. Let $F:\,\Co1(\Om)\times\mA\to [0,\infty]$ be the  functional in \eqref{mfonC1} with $\Lef:\,\Om\times\Rn\to [0,\infty]$ a Borel measurable function satisfying
\begin{equation}\label{felocsom}
\text{ for each } \xi\in\Rn,\,\Lef(\cdot,\xi)\in L^1_{\rm loc}(\Om)\,
\end{equation}
%\begin{equation}\label{fePixlocsom}
%\text{ for each } \xi\in\Rn,\,\Lef(\cdot,\Pi_\cdot\,\xi)\in L^1_{\rm loc}(\Om_X)\,,
%\end{equation}

and
\begin{equation}\label{feconv}
\Lef(x,\cdot):\,\Rn\to [0,\infty)\text{ convex for a.e. }x\in\Om\,.
\end{equation}
%and \eqref{f0Vx}.
%\begin{equation}\label{f0Vx}
%\begin{split}
%\Lef(x,\xi)=\,\Lef(x, \Pi_x(\xi))\quad\forall\,\xi\in\Rn\,\mbox{and}\, \text{ a.e. }x\in\Om,
%\end{split}
%\end{equation}
%where $\{V_x:x\in\Om_X\}$ is the distribution of $m$-planes in $\Rn$ defined in Proposition \ref{decompRnbyVx} and $\Pi_x:\,\Rn\to V_x$ denotes the projection of $\Rn $  on $V_x$ in \eqref{Pix}.  
Define $f:\,\Om\times\Rm\to [0,\infty)$ as
\begin{equation}\label{deffbyfe}
f(x,\eta):=
\begin{cases}\,\Lef(x,L^{-1}(x)(\eta))&\text{ if }(x,\eta)\in\Om_X\times\Rm\\
0&\text{ otherwise}
\end{cases}
\,,
\end{equation}
where $L^{-1}:\,\Om_X\to\Lap(\Rm,\Rn)$ is the map in \eqref{L-1}.
Then, $f$ is a Borel measurable function satisfying
\begin{equation}\label{fconvex2}
f(x,\cdot):\,\Rm\to [0,\infty)\text{ convex for a.e. }x\in\Om\,.
\end{equation}
Moreover, 
\begin{equation}\label{represFwrtX}
\begin{split}
F(u,A)&= \int_A \Lef(x,Du)\,dx\\
&=\,\int_A f(x,Xu)\,dx\quad\forall A\in\mA,\,u\in \Co1(A)\,
\end{split}
\end{equation}
if and only if
\begin{equation}\label{f0Vx}
\begin{split}
\Lef(x,\xi)=\,\Lef(x, \Pi_x(\xi))\quad \mbox{for a.e.}\ x\in\Om,\ \forall\,\xi\in\Rn\,,
\end{split}
\end{equation}
where $\{V_x:x\in\Om_X\}$ is the distribution of $m$-planes in $\Rn$ defined in Proposition \ref{decompRnbyVx} and $\Pi_x:\,\Rn\to V_x$ denotes the projection of $\Rn $  on $V_x$ in \eqref{Pix}. \\
In addition, the function $f$ for which \eqref{represFwrtX} holds is unique, that is,  if there exists another Borel measurable function  $f^*:\,\Om\times\Rm\to [0,\infty)$ satisfying $f^*(x,\cdot):\,\Rm\to [0,\infty)$ convex a.e. $x\in\Om$ and \eqref{represFwrtX}  holds, then $f(x,\eta)=\,f^*(x,\eta)$ for a.e. $x\in\Om$ and $\eta\in\Rm$.
\end{teo}

\begin{oss}\label{nouniqrepres} If the $X$-gradient does not satisfy (LIC) condition, the uniqueness of representation \eqref{represFwrtX} may trivially fail. For instance, let $X=(X_1,X_2):=(\partial_1,\,0)$ be the family of vector fields on $\Omega=\,\bbR^2$ and let $f(\eta):=\,\eta_1^2+\,g(\eta_2)$ and $f^*(\eta):=\,\eta_1^2+\,g^*(\eta_2)$ for each $\eta=(\eta_1,\eta_2)\in\bbR^2$, where $g,\,g^*:\,\bbR\to[0,\infty)$ are convex functions satisfying $g(0)=\,g^*(0)=\,0$, but $g\neq\,g^*$. Then it clear that  $f$ and $f^*$ are integrand functions of the same functional $F$ defiined in \eqref{represFwrtX} , even though $f\neq f^*$.
\end{oss}

\begin{oss}\label{Lefnotindipx} Notice that, in the case $m=n$ and $X$ satisfies (LIC) on $\Omega$, condition \eqref{f0Vx}  always holds, since, by Lemma \ref{helplemma} (iii), $\Pi_x\equiv {\rm Id}_n$. 
%Notice also  that  a Euclidean integrand must satisfy compatibility condition \eqref{f0Vx}. This implies that, if $m<\,n$, since $N_x\neq \{0\}$, $\Pi_x$ must depend on $x$ for each $x\in\Omega_X$.Thus, in this case, a Euclidean integrand cannot be independent of $x$. 
\end{oss}

\begin{proof}\noindent {\bf 1st step.} Let us prove that $f$ is Borel measurable.
Let $\Psi:\,\Omega_X\times\Rm\to\Omega_X\times\Rn$ denote the map 
\[
\Psi(x,\eta): =(x,L^{-1}(x)(\eta))\quad\text{if }(x,\eta)\in\Om_X\times\Rm.
\]
By Proposition \ref{decompRnbyVx}, $\Psi$ is continuous, then it is also Borel measurable.  Since $\Lef:\,\Omega\times\Rn\to [0,\infty]$ is Borel measurable, the composition $f\,=\,\Lef\circ\Psi:\Omega_X\times\Rm \to [0,\infty]$ is still Borel measurable.

To prove \eqref{fconvex2} it is sufficient to notice that
\[
f(x,\cdot)=\,\Lef(x,\cdot)\circ L^{-1}(x)\quad\forall\,x\in\Om_X\,
\]
indeed $\Lef(x,\cdot):\,\Rn\to [0,\infty)$ is convex for a.e. $x\in\Om$ and $L^{-1}(x):\,\Rm\to\Rn$ is linear for $x\in\Om_X$.\\

{\bf 2nd step.} Let us prove the uniqueness of representation in \eqref{represFwrtX}. Assume that
\begin{equation}\label{eqffstarfe}
\begin{split}
\int_A f(x,Xu)\,dx&=\,\int_A f^*(x,Xu)\,dx\\
&=\,\int_A \Lef(x,Du)\,dx\quad\forall\,u\in \Co1(A),\,A\in\mA\,
\end{split}
\end{equation}
for given  Borel measurable functions  $f,\,f^*:\,\Om\times\Rm\to [0,\infty)$, with $f(x,\cdot),f^*(x,\cdot):\,\Rm\to [0,\infty)$ convex a.e. $x\in\Om$.
Let us choose as 
\begin{equation}\label{uxi}
u(x)=\,u_\xi(x):=\langle\xi,x\rangle\quad x\in\Rn\,,
\end{equation} 
for fixed $\xi\in\bbQ^n$, in the previous equality.
%and let us define the Borel measures on $\mB(\Omega)$
%\[
%\begin{split}
%&\mu(A):=\,\int_A f(x,C(x)\,\xi)\,dx,\,\mu^*(A):=\,\int_A f^*(x,C(x)\,\xi)\,dx,
%\\&\,\mu_e(A):=\,\int_A f_e(x,\xi)\,dx\text{ if }A\in\mB(\Omega)\,.
%\end{split}
%\] 
By \eqref{eqffstarfe} and \eqref{felocsom}, it follows that  the functions
\[
\Omega\ni y\mapsto f(y,C(y)\xi)\text{ and }\Omega\ni y\mapsto f^*(y,C(y)\xi)\text{ are in }L^1_{\rm loc}(\Om)\,.
\]
%they are actually Radon measures on $\Om$ and they agree on  $\mA$.
%\[
%\mA_X:=\,\left\{A\in\mA:\,A\subset\Omega_X\right\}\,.
%\] 
%By a well-known result of approximation for Radon measures by open sets, we can infer that they agree on the whole $\mB(\Omega)$. 
Choosing $A=B(x,r)$ in \eqref{eqffstarfe}, by Lebesgue's differentiation theorem, 
we get  that there exists a negligible set $\mathcal N_\xi\subset\Om$ such that $\forall\,x\in\Om\setminus\mathcal N_\xi$
\begin{equation}\label{coincffstar}
\begin{split}
f(x,L_x(\xi))&=\,f(x,C(x)\xi)=\,f^*(x,C(x)\xi)\\
&=\,f^*(x,L_x(\xi))\,.
\end{split}
\end{equation}
If $\mathcal N:=\,\cup_{\xi\in\bbQ^n}\mathcal N_\xi$, then  \eqref{coincffstar} holds for each $x\in\Om\setminus\mathcal N$ and $\xi\in\bbQ^n$.  Since, for each $x\in\Om\setminus\mathcal N$, $f(x,\cdot),f^*(x,\cdot):\,\Rm\to[0,\infty)$ are continuous, it follows that  \eqref{coincffstar}  holds for each $x\in\Om\setminus\mathcal N$ and $\xi\in\Rn$.
Being the map $L_x:\,\Rn\to\Rm$ onto, we get the desired conclusion.\\

\noindent {\bf 3nd step.} Let us assume \eqref{f0Vx}. To prove \eqref{represFwrtX} it is sufficient to prove that, for each $A\in\mA$, $u\in\Co1(A)$
\begin{equation}\label{f=f0}
f(x,Xu(x))=\,\Lef(x,Du(x))\quad\text{a.e. }x\in\Om\,.
\end{equation}
Given $A\in\mA$ and $u\in\Co1(A)$, let us recall that
\[
Xu(x)=\,C(x)Du(x)\quad\forall\,x\in A\,.
\]
Thus, by \eqref{f0Vx}, Lemma \ref{helplemma} (iii) and the definition of $V_x$ , a.e. $x\in\Om$, if $v_x:=Du(x)$
\begin{equation}\label{ifffe}
\begin{split}
f(x, Xu(x))&=\,f(x,C(x)v_x)=\,f(x,L_x(\Pi_x(v_x)))\\
&=\,\Lef(x,L_x^{-1}(L_x(\Pi_x(v_x)))=\,\Lef(x,\Pi_x(v_x))\\
&=\,\Lef(x,v_x)=\,\Lef(x,Du(x))\,
\end{split}
\end{equation}
and \eqref{f=f0} follows. On the other hand, let us assume that for every $A\in\mA$ and $u\in \Co1(A)$ 
\[
\int_A \Lef(x,Du)\,dx=\,\int_A f(x,Xu)\,dx
\]
where $f$ is the function in \eqref{deffbyfe}.
By \eqref{ifffe}, for every $A\in\mA$ and $u\in\Co1 (A)$,
\[
f(x,Xu(x))=\,\Lef(x,\Pi_x(Du(x)))\quad\forall\,x\in A\,,
\]
which implies
\[
\int_A f(x,Xu(x))\,dx=\, \int_A \Lef(x,\Pi_x(Du(x)))\,dx\,.
\]
Thus, for every $A\in\mA$ and $u\in \Co1(A)$,
\[
\int_A \Lef(x,\Pi_x (Du(x)))\,dx=\int_A \Lef(x,Du)\,dx
\]
and the conclusion now follows by proceeding as in the second step of the proof.
\end{proof}
\begin{oss}\label{represFwrtXW1p} Observe that \eqref{f=f0} actually holds for each $u\in W^{1,p}(A)$. As a consequence, \eqref{represFwrtX} holds for each $A\in\mA$ and $u\in W^{1,p}(A)$.
\end{oss}

\subsection{Integral representation for local functionals with respect to vector fields}
Let us recall, for reader's convenience, some notation about set functions on $\mA$ and  local functionals on $L^p(\Om)\times\mA$, according to \cite{DM}. Let $\Om\subset\Rn$ be an open set. 
\begin{defi} Let $\alpha:\,\mA\to [0,\infty]$ be a set function.
We say that:
\begin{itemize}
\item[(i)] $\alpha$ is {\it increasing} if $\alpha(A)\le\, \alpha(B)$, for each $A,\,B\in\mA$ with $A\subseteq B$;
\item[(ii)] $\alpha$ is {\it inner regular } if
\[
\alpha(A)=\,\sup\left\{\alpha(B):\,B\in\mA,\,B\Subset A\right\}\,\text{for each}\ A\in\mA;
\]
%\item[(iii)] $\alpha$ is {\it outer regular } if
%\[
%\alpha(A)=\,\inf\left\{\alpha(B):\,B\in\mA,\,A\Subset B\right\}\,\text{for each}\ A\in\mA,
%\]
%with the usual convention that $\inf\emptyset=\,\infty$;
\item[(iii)] $\alpha$ is {\it subadditive }if $\alpha(A)\le\,\alpha(A_1)+\alpha(A_2)$ for every $A,\,A_1,\,A_2\in\mA$ with $A\subset A_1\cup A_2$;
\item[(iv)] $\alpha$ is {\it superadditive }if $\alpha(A)\ge\,\alpha(A_1)+\alpha(A_2)$ for every $A,\,A_1,\,A_2\in\mA$ with $A_1\cup A_2\subseteq A$ and $A_1\cap A_2=\emptyset$;
\item[(v)] $\alpha$ is a {\it measure }if there exists a Borel measure $\mu:\,\mathcal B(\Om)\to [0,\infty]$ such that $\alpha(A)=\,\mu(A)$ for every $A\in\mA$.
\end{itemize}
\end{defi}
\begin{oss}\label{charmeas}Let us recall that, if  $\alpha:\,\mA\to [0,\infty]$ is an increasing  set function, then it is a measure if and only if it is subadditive, superadditive and inner regular (see \cite[Theorem 14.23]{DM}).
\end{oss}
\begin{defi}Let
\[
F:\,L^p(\Om)\times\mA\to [0,\infty]\,.
\]
We say that:
\begin{itemize}
\item[(i)] $F$ is {\it increasing} if, for every $u\in L^p(\Om)$, $F(u,\cdot):\,\mA\to [0,\infty]$ is increasing as set function;
\item[(i)] $F$ is {\it inner regular} (on $\mA$) if it is increasing and, for each $u\in L^p(\Om)$, $F(u,\cdot):\,\mA\to [0,\infty]$ is innner regular as set function;
\item[(iii)] $F$ is a {\it measure}, if for every $u\in L^p(\Om)$, $F(u,\cdot):\,\mA\to [0,\infty]$ is a measure as set function;
\item[(iv)] $F$ is {\it local} if
\[
F(u,A)=\,F(v,A)\,
\]
for each $A\in\mA$, $u,v\in L^p(\Om)$ such that $u=v$ a.e. on $A$;
\item[(v)] $F$ is {\it lower semicontinuous} (lsc), if for every $A\in\mA$, $F(\cdot,A):\,L^p(\Om)\to [0,\infty]$ is lower semicontinuous.
\end{itemize}
\end{defi}
\begin{teo}\label{DMThm201ext} Let $\Om\subset\Rn$ be a bounded open set and assume that $X$ satisfies (LIC) on $\Omega$. Let $p>\,1$ and 
\[
F:\,L^p(\Om)\times\mA\to [0,\infty]
\]
be an increasing functional satisfying the following properties:
\begin{itemize}
\item[(a)] $F$ is local;
\item[(b)] $F$ is a measure;
\item[(c)] $F$ is lsc;
\item[(d)] $F(u+c,A)=\,F(u,A)$ for each $u\in L^p(\Om)$, $A\in\mA$ and $c\in\bbR$;
\item[(e)] there exist a non negative function $a\in L^{1}_{\rm loc}(\Omega)$ and a positive constant $b$ such that
\[
0\le\,F(u,A)\le\,\int_A\left(a(x)+b\,|Xu(x)|^p\right)\,dx
\]
for each $u\in\Co1(A)$, $A\in\mA$.
%for each $u\in W^{1,p}(A)$, $A\in\mA$.
%\item[(f)] for a.e. $x\in\Om$, there exist
%\[
%\lim_{r\to 0^+}\frac{F(u_\xi,B(x,r))}{|B(x ,r)|}=\, \lim_{r\to 0^+}\frac{F(u_{\xi_{V_x}},B(x,r))}{|B(x ,r)|}\text{ for each }\xi\in\Rn\,,
%\]
%where $u_\xi$ is the function in \eqref{uxi} and $\xi_{V_x}=\,\Pi_x(\xi)$ is the vector in \eqref{Pix}. 
\end{itemize}
Then, there exists a Borel function $f:\,\Omega\times\Rm\to [0,\infty]$ such that:
\begin{itemize}
\item[(i)] for each $u\in L^p(\Om)$, for each $A\in\mA$ with $u|_A\in W^{1,p}_{X;{\rm loc}}(A)$, we have
\[
F(u,A)=\,\int_A f(x,Xu(x))\,dx\,;
\]
\item[(ii)] for a.e. $x\in\Om$, $f(x,\cdot):\,\Rm\to [0,\infty)$ is convex;
\item[(iii)] for a.e. $x\in\Om$,
\[
0\le\,f(x,\eta)\le\,a(x)+\, b\, |\eta|^p\quad\forall\,\eta\in\Rm\,.
\]
%\[
%0\le\, f\left(x,\langle X(x),\xi\rangle\right)\le\,a(x)+\,b\,\left|\langle X(x),\xi\rangle\right|^p\text{ for each }\xi\in\Rn
%\]
%where
%\begin{equation}\label{Xdotxi}
%\langle X(x),\xi\rangle:=\,\left(\langle X_1(x),\xi\rangle,\dots, \langle X_m(x),\xi\rangle\right)=\,C(x)\xi\,.
%\end{equation}
\end{itemize}
\end{teo}
%%%%%%%%%%%%%%%%%%%%%%%%%%%%%%%%%%%%%%%%%%%%
In order to prove Theorem \ref{DMThm201ext}, we need two auxiliary key lemmas.

The former is well-known (see, for instance, \cite[Theorem 12.1]{Ro}). Let us recall that an {\it affine function} $\varphi:\,\Rn\to\bbR$ is a function 
\[
\varphi(\xi)=\,\langle z,\xi\rangle+\,k\quad\forall\,\xi\in\Rn\,,
\]
for a suitable $z\in\Rn$ and $k\in\bbR$.
\begin{lem}\label{supaffine} Let $g:\,\Rn\to\bbR$ be a convex function. Then
\[
g(\xi)=\,\sup\left\{\varphi(\xi):\,\varphi\text{ affine, }\varphi(\xi)\le\,g(\xi)\,\quad\forall\,\xi\in\Rn\right\}\,.
\]
\end{lem}
The latter will turn out to be a key result through the paper and provides when a Euclidean integrand can be represented as an integrand respect to $X$-gradient.

\begin{lem}\label{keylemma} Let $\Lef:\,\Omega\times\Rn\to [0,\infty]$ be a Borel measurable function. Suppose that
\begin{itemize}
\item[(i)] $\text{for a.e. $x\in\Om$, $f_e(x,\cdot):\,\Rn\to [0,\infty)$ is convex}$;
\item[(ii)] there exist a non negative function $a\in L^{1}_{\rm loc}(\Omega)$ and a positive constant $b$ such that for a.e. $x\in\Om$
\[
\Lef(x,\xi)\le\,a(x)+b|C(x)\xi|^p\quad\forall\,\xi\in\Rn\,,
\]
where $C(x)$ denotes the coefficient matrix of $X$-gradient in \eqref{coeffmatrvf}.
\end{itemize}
Then, $\Lef$ satisfies \eqref{f0Vx}.
\end{lem}
\begin{proof}Let us prove that,  for a.e. $x\in\Om$,
\begin{equation}\label{identcruc}
\Lef(x,\xi_{N_x}+\zeta)=\,\Lef(x,\zeta)\quad\forall\,\xi,\,\zeta\in\Rn\,,
\end{equation}
according to notation in section 3.2.
Notice  that \eqref{identcruc} is equivalent to \eqref{f0Vx}, that is, for a.e. $x\in\Om$
\[
\Lef(x,\xi)=\,\Lef(x,\xi_{V_x})\quad\forall\,\xi\in\Rn\,.
\]
By our assumptions, we can assume that, for a.e. $x\in\Om$,  $g:=\Lef(x,\cdot):\,\Rn\to [0,\infty)$ is a convex function and (ii) holds with $a=\,a(x)\in [0,\infty)$. Let $\varphi:\,\Rn\to\bbR$ be affine with $\varphi(\xi)=\,\langle z,\xi\rangle+k$ and $\varphi(\xi)\le\,g(\xi)$ for each  $\xi\in\Rn$.  Let us prove that
\begin{equation}\label{affineconst}
\langle z,v\rangle=\,0\quad\forall\,v\in N_x\,.
\end{equation}
Let $v\in N_x\setminus\{0\}$ be given, then  also $t\,v\in N_x$ for each $t\in\bbR$. In particular, $C(x)tv=\,0$ for each $t\in\bbR$. Then, by (ii)
\[
\varphi(t\,v)=\,t\,\langle z,v\rangle+k\le\,g(tv)\le\,a\quad\forall\,t\in\bbR\,.
\]
The previous inequality implies \eqref{affineconst}. From \eqref{affineconst}, we get that
\begin{equation}
\begin{split}
\varphi(\xi_{N_x}+\zeta)&=\,\langle a,\xi_{N_x}+\zeta\rangle+b\\
&=\,\langle a,\zeta\rangle+b=\,\varphi(\zeta)\quad\forall\,\xi,\zeta\in\Rn\,.
\end{split}
\end{equation}
From Lemma \ref{supaffine}, \eqref{identcruc} follows.
\end{proof}
%%%%%%%%%%%%%%%%%%%%%%%%%%%%%%%%%%%%%%%%%%%%%%%%%%%%%%%%%%%%%%%%%%
\begin {proof}[Proof  Theorem \ref{DMThm201ext}] Let us first observe that inequality in assumption (e) can be extended to each $u\in W_X^{1,p}(A)$, $A\in\mA$.  Let us  recall that, if $A\in\mA$, by Proposition \ref{Proposition 1.2.2}, given $(\rho_\ep)_\eps$ a family of mollifiers, then, for each $u\in \W1p(A)$, denoting by $\bar u$ its extension to $\mathbb{R}^n$ being $0$ outside $\Omega$, if 
\[
u_\eps(x):=\bar u*\rho_\eps(x)\quad x\in\Rn\,,
\]
for each $A'\in\mA$ with $A'\Subset A$, we have
\begin{equation}\label{convuepsLp}
u_\eps\rightarrow u\text{ in }L^p(\Om)\,;
\end{equation}
\begin{equation}\label{convuepsWXp}
u_\eps|_{A'}\in \W1p(A')\text{ and } u_\eps\rightarrow u\text{ in }\W1p(A')\,.
\end{equation}
Let $u\in L^p(\Om)$ be such that $u|_A\in W^{1,p}(A)$ for some $A\in\mA$. For each $A'\Subset A$, by assumption (c), \eqref{convuepsLp} and \eqref{convuepsWXp}, it follows that
\[
\begin{split}
F(u,A')&\le\,\liminf_{\eps\to 0}F(u_\eps,A')\le \lim_{\eps\to 0}\left(\int_{A'}\left(a(x)+b\,|Xu_\eps(x)|^p\right)\,dx\right)
\\&=\,\int_{A'}\left(a(x)+b\,|Xu(x)|^p\right)\,dx.
\end{split}
\]
Since $F(u,\cdot)$ is a measure, it is also inner regular (see Remark \ref{charmeas}). Thus, taking the supremum on all $A'\in\mA$ with $A'\Subset A$, we get the desired conclusion.
%%%%%%%%%%%%%%%%%%%%%%%%%%%%%%%%
We will now divide the proof in three steps.

{\bf 1st step.} Let us first prove that there exists an integral representation of $F$ with respect to a Euclidean integrand, that is, there exists a Borel function $\Lef:\,\Om\times\Rn\to [0,\infty]$ and a positive constant $b_2$ such that 
\begin{equation}\label{euclrepresF}
F(u,A)=\,\int_A\Lef(x,Du)\,dx\,,
\end{equation}
for each $u\in L^p(\Om)$, $A\in\mA$ with $u|_A\in\ W^{1,p}_{\rm loc}(A)$;
\begin{equation}\label{f0convex}
\text{ for a.e. $x\in\Om$, $f_e(x,\cdot):\,\Rn\to [0,\infty)$ is convex};
\end{equation} 
\begin{equation}\label{controlf0}
\text{for a.e. $x\in\Om$, $0\le\,\Lef(x,\xi)\le\,a(x)+\,b_2\, |\xi|^p\quad\forall\,\xi\in\Rn$};
\end{equation}
\begin{equation}\label{f0VxThm201ext}
\text{\eqref{f0Vx} holds, that is, for a.e. $x\in\Om$ $\Lef(x,\xi)=\,\Lef(x,\Pi_x(\xi))\quad\forall\,\xi\in\Rn$}.
\end{equation}
%where $\Pi_x$ is the projection map in \eqref{Pix}\,.

By \eqref{represXbyD}, if $u\in W^{1.p}(\Om)$, then, for a.e. $x\in\Om$, we have that

\begin{equation}\label{control|Xu|by|Du|}
\left|Xu(x)\right|^p\le\,\sup_{x\in\Om}\norma{C(x)}^p\,|Du(x)|^p=b_2\,|Du(x)|^p\,,
\end{equation}
with $b_2<\,\infty$, since the coefficients of  $X$-gradient are Lipschitz on $\Om$. By \eqref{control|Xu|by|Du|} and assumption (e), it follows that
\begin{equation}\label{tildee}
0\le\,F(u,A)\le\,\int_A\left(a(x)+b_2|Du(x)|^p\right)\,dx\,,
\end{equation}
for each $u\in W^{1.p}(\Om)$, for every $A\in\mA$ . Therefore by (a), (b), (c), (d) and \eqref{tildee}, by applying \cite[Theorem 20.1]{DM}, there exists a Borel function $\Lef:\,\Om\times\mathbb{R}^n\to [0,\infty]$ satisfying \eqref{euclrepresF}, \eqref{f0convex} and \eqref{controlf0}.  Observe now that, by \eqref{euclrepresF} and assumption (e), if $u=u_\xi$, if follows that, for each $x\in\Rn$,
\[
\int_A\Lef(x,\xi)\,dx\le\,\int_A(a(x)+b|C(x)\xi|^p)\,dx\quad\forall\,A\in\mA.
\]
From this integral inequality, arguing as in section 3.2, we can infer the pointwise inequality, that is,  there exists a negligible set $\mathcal N\subset\Om$, such that, for each $x\in\Om\setminus\mathcal N$,
\begin{equation}\label{keyineqLef}
\Lef(x,\xi)\le\,a(x)+b|C(x)\xi|^p\quad\forall\,\xi\in\Rn\,,
\end{equation}
From \eqref{f0convex}, \eqref{keyineqLef} and Lemma \ref{keylemma}, \eqref{f0VxThm201ext} holds.
{\bf 2nd step.} Let us prove that there exists a Borel function $f:\,\Om\times\Rm\to [0,\infty]$  
such that 
\begin{equation}\label{represFC1}
F(u,A)=\,\int_A f(x,Xu)\,dx\,,
\end{equation}
for each $A\in\mA$, $u\in \Co1(A)$ satisfying claims (ii) and (iii).
%%%%%%%%%
%\begin{equation}\label{fconvex}
%\text{ for a.e. $x\in\Om$, $f(x,\cdot):\,\Rm\to [0,\infty)$ is convex};
%\end{equation} 
%and 
%\begin{equation}\label{controlf}
%\text{for a.e. $x\in\Om$, $0\le\,f(x,\eta)\le\,a(x)+\,b\,|\eta|^p\quad\forall\,\eta\in\Rm$}
%%&0\le\,f(x,\eta)\le\,(a(x)+\,b\,|\eta|^p)\quad\forall\,\eta\in\Rn.
%\end{equation}
%where $a\in L^1_{loc}$ and $b>0$ are as in (e). 
%%%%%%
By \eqref{f0convex}, \eqref{controlf0} and \eqref{f0VxThm201ext}, we can apply Theorem \ref{reprvfeg} and  \eqref{represFC1} follows at once  with $f:\,\Om\times\Rm\to [0,\infty]$  defined as in \eqref{deffbyfe}, which also satisfies claim (ii). %\eqref{fconvex} .

 From assumption (e) and \eqref{represFC1}  with $u=u_\xi$, it follows that
\[
0\le\,\int_A f(y,C(y)\xi)\,dy\le\,\int_A\left(a(y)+ b\, |C(y)\xi|^p\right)\,dy\quad\text{ for each }A\in\mA,\,\xi\in\Rn\,.
\]
Taking $A=B(x,r)$, applying Lebesgue's differentiation theorem and arguing as before, from the previous inequality, we can get the following pointwise estimate: for a.e. $x\in\Om$ it holds that
\[
0\le\,f(x,C(x)\xi)\le\,a(x)+\,b\,|C(x)\xi|^p\quad\forall\,\xi\in\Rn\,.
\]
Observe now that, by (LIC), for a.e. $x\in\Om$, the map $L_x:\,\Rn\to\Rm$, $L_x(\xi):=\,C(x)\xi$, is surjective. Then claim (iii) also follows.

{\bf 3rd step.} Let us prove that the integral representation in \eqref{represFC1} can be extended to functions $u\in W^{1,p}_{X,{\rm loc}}(A)$. Therefore claim (i) will follow.

Let us begin to observe that, given $A\in\mA_0$, the functional
\begin{equation}\label{strongcontF}
\W1p(A)\ni u\mapsto\int_Af(x,Xu)\,dx\text{ is (strongly) continuous}.
\end{equation}
Indeed, since for a.e. $x\in\Om$, $f(x,\cdot):\,\Rm\to [0,\infty)$ is continuous and claim (iii) holds, we can apply the  Carath\'eodory continuity theorem (see, for instance, \cite[Example 1.22]{DM}). 

%Let us also recall that, by Proposition \ref{Proposition 1.2.2}, if $(\rho_\ep)_\eps$ is a family of mollifiers, then, for each $u\in \W1p(\Om)$, denoting by $\bar u$ its extension to $\mathbb{R}^n$ being $0$ outside $\Omega$, if 
%\[
%u_\eps(x):=\bar u*\rho_\eps(x)\quad x\in\Rn\,,
%\]
%for each $\Om'\Subset\Om$, we have
%\begin{equation}\label{convuepsLp}
%u_\eps\rightarrow u\text{ in }L^p(\Om)\,;
%\end{equation}
%\begin{equation}\label{convuepsWXp}
%u_\eps\rightarrow u\text{ in }\W1p(\Om')\,.
%\end{equation}

Let $u\in W^{1,p}_{X}(\Om)$ and let $A,\,A'\in\mA$ with $A'\Subset A$.  Since $F(\cdot,A'):\,L^p(\Om)\to [0,\infty]$, by \eqref{convuepsLp}, it follows that
\[
\begin{split}
F(u,A')\le\,\liminf_{\eps\to 0^+}F(u_\eps,A')=\,\lim_{\eps\to 0^+}\int_{A'}f(x,Xu_\eps)\,dx=\,\int_{A'}f(x,Xu)\,dx\,.
\end{split}
\]
As $F$ is a measure, taking the limit as $A'\uparrow A$, we get
\begin{equation}\label{leftineqreprF}
F(u,A)\le\,\int_{A}f(x,Xu)\,dx\,,
\end{equation}
for every $u\in\W1p(\Om)$, for each $A\in\mA$.

Let us fix $w\in\W1p(\Om)$ and let us consider the functional $G:\,L^p(\Om)\times\mA\to [0,\infty]$
\[
G(u,A):=\,F(u+w,A)\,.
\]
It is easy to show that $G$ still satisfies assumptions (a)-(e). %The proof of properties (a)-(e) are quite standard. 
Thus, by the second step,  there exists a Borel function $g:\,\Om\times\Rm\to $ satisfying claims (ii) and (iii) with $f\equiv g$, for suitable $a\in L^1_{\rm loc}(\Om)$ and $b>\,0$ such that

\begin{equation}\label{represGC1}
G(u,A)=\,\int_A g(x,Xu)\,dx\,,
\end{equation}
for each $A\in\mA$, $u\in \Co1(A)$ and 
\begin{equation}\label{leftineqreprG}
G(u,A)\le\,\int_{A}g(x,Xu)\,dx\,,
\end{equation}
for every $u\in\W1p(\Om)$, for each $A\in\mA$. Moreover, arguing as in \eqref{strongcontF}, one can prove that, for each $A\in\mA_0$, the functional
\begin{equation}\label{strongcontG}
\W1p(A)\ni u\mapsto\int_Ag(x,Xu)\,dx\text{ is (strongly) continuous}.
\end{equation}
Let
\[
w_\eps:=\,\bar w*\rho_\eps:\,\Rn\to\bbR\,
\]
and fix $A\in\mA$. Then, for every $A'\in\mA$ with $A'\Subset A$, as $\eps\to 0^+$,
\[
w_\eps\to w\text{ in }L^p(\Om)\text{ and  }w_\eps\to w\text{ in }\W1p(A')\,.
\]
Thus, by \eqref{strongcontF}, \eqref{leftineqreprF}, \eqref{represGC1}, \eqref{leftineqreprG}, \eqref{strongcontG} we obtain
\[
\begin{split}
\int_{A'}g(x,0)\,dx&=\,G(0,A')=\,F(w,A')\le\,\int_{A'}f(x,Xw)\,dx\\
&=\,\lim_{\eps\to 0^+}\int_{A'}f(x,Xw_\eps)\,dx=\,\lim_{\eps\to 0^+}F(w_\eps,A')\\
&=\,\lim_{\eps\to 0^+}G(w_\eps-w,A')=\,\lim_{\eps\to 0^+}\int_{A'}g(x,Xw_\eps-Xw)\,dx\\
&=\,\int_{A'}g(x,0)\,.
\end{split}
\]
This implies that
\begin{equation*}
F(w,A')=\,\,\int_{A'}f(x,Xw)\,dx\text{ for each }A'\in\mA\text{ with }A'\Subset A\,.
\end{equation*}
Taking the limit as $A'\uparrow A$ in the previous identity, we get that

\begin{equation}\label{intrepresW1pOm}
F(w,A)=\,\,\int_{A}f(x,Xw)\,dx\text{ for each }w\in\W1p(\Om)\text{ and }A\in\mA\,.
\end{equation}
If $u\in L^p(\Om)$, $A\in\mA$ and $u|_A\in W_{X;{\rm loc}}^{1,p}(A)$ then, for every $A'\in\mA$ with $A'\Subset A$, by
Remark \ref{extpropoW1pXloc}, there exists $w\in\W1p(\Om)$ such that
\[
u|_{A'}=\,w|_{A'}\,.
\]
Since $F$ is local, by \eqref{intrepresW1pOm}, we obtain that
\[
F(u,A')=\,F(w,A')=\,\int_{A'}f(x,Xw)\,dx=\,\int_{A'}f(x,Xu)\,dx\,.
\]
Taking the limit as $A'\uparrow A$ we get
\[
F(u,A)=\,\int_{A}f(x,Xu)\,dx\,,
\]
which concludes the proof. 
%Claims (ii) and (iii) follow from \eqref{fconvex} and \eqref{controlf}.
\end{proof}

\begin{counterex}\label{counterex} {\rm If $X$ agrees with the Euclidean gradient (Example \ref{exvf} (i)), there are well-known examples that, dropping one of the assumptions among (a)-(e) in Theorem \ref{DMThm201ext}, then the conclusion may fail (see, for instance, \cite{B}). Let $X$ be the Heisenberg vector fields in $\bbR^3$ (Example \ref{exvf} (iii)), let $\Omega\subset\bbR^3$ be a bounded open set containing the origin and $p=2$. Then we give an instance that, dropping assumption (e), the conclusion of   Theorem \ref{DMThm201ext} may  fail. Let $F:\,L^2(\Om)\times\mA\to [0,\infty]$ be the local functional defined as
\[
F(u,A):=\,\begin{cases} \int_A |Du|^2\,dx&\text{ if }u\in W^{1,2}(A)\\
\infty&\text{ otherwise }
\end{cases}
\,.
\]
Then, it is clear that $F$ satisfies (a)-(d). Let us prove that functional $F$ cannot satisfy claim (i).  Indeed, by contradiction, if there is some integrand $f:\,\Omega\times\bbR^2\to [0,\infty]$ for which (i) holds, then, by Theorem \ref{reprvfeg}, the compatibility condition \eqref{f0Vx} must be satisfied, that is, 
\[
|\xi|^2=\,\Lef(x,\xi)=\,f(x,C(x)\xi)=\,\Lef(x,\Pi_x(\xi))=\,|\Pi_x(\xi)|^2
\]
for a.e. $x\in\Om\,,\forall\, \xi\in\bbR^3$. Since, by Lemma \ref{helplemma} (iii), function $\Omega\ni x\mapsto\Pi_x(\xi)$ is continuous, the previous identity must hold for each $x\in\Om$ and $\xi\in\bbR^3$.  Let $x=0$, then a simple calculation yields that $\Pi_0(\xi)=\,(\xi_1,\xi_2,0)$  for each $\xi=\,(\xi_1,\xi_2,\xi_3)\in\bbR^3$. Thus, if we choose $\xi=\,(0,0,1)$, the previous identity is not satisfied and then we have a contradiction. This example also shows that the correspondence which maps integrand $f(x,\eta)$  to Euclidean integrand $\Lef (x,\xi):=\,f(x,C(x)\xi)$ cannot be reversed.
 
 }
\end{counterex}
%%%%%%%%%%%%%%%%%%%%%%%%%%%%%%%%%%%%%%%%%%%%%%%%%%%

%%%%%%%%%%%%%%%%%%%%%%%%%%%%%%%%%%%%%%%%%%%%%%%%%%%%%%%%%%%%%%%%%%%%%
\section{$\Gamma$-convergence for integral functionals depending on vector fields} In this section we are going  to show some results concerning $\Gamma$-convergence of integral functionals depending on vector fields, in the strong and weak topology of $W^{1,p}_X(\Om)$ and in the strong one of $L^p(\Om)$. In particular, we will prove a $\Gamma$-compactness result for a class of integral functionals depending on vector fields with respect to $L^p(\Omega)$-topology (see Theorem \ref{mainthm}).

Let us first recall some notions and results concerning $\Gamma$-convergence theory, which are contained in the fundamental monograph  \cite{DM} and to which we will refer through this section. We also recommend monograph \cite{Bra} as exastuive account on this topic, containing also  interesting applications of $\Gamma$-convergence.
 %%%%%%abstractGammaconv
 
Let $(X,\tau)$ be a topological space and let $(F_h)_h$ be a sequence of functionals from the space $(X,\tau)$ to $\bar\bbR$. 
%We pose for every $x \in X$
%\[
%s -\Gamma(\tau) \liminf_{h\to\infty} F_h(x) = \inf\left\{\liminf _{h\to\infty}F_h(x_h) : x_h\to x\right\}\,.
%\]
%\[
%s -\Gamma(\tau) \limsup_{h\to\infty} F_h(x) = \inf\left\{\limsup_{h\to\infty}F_h(x_h) : x_h\to x\right\}\,. 
%\]
%We call the first member respectively {\it sequential $\Gamma$-lower limit} and {\it sequential $\Gamma$-upper limit} of sequence $(F_h)_h$ in relation to topology $\tau$. 
Let $\mathcal U(x)$ be the family of open neighborhoods of $x\in X$. Then
we pose for every $x\in X$
\[
(\Gamma(\tau)-\liminf_{h\to\infty} F_h)(x) = \sup_{U\in\mathcal U(x)}\liminf_{h\to\infty}\inf_{U}F_h\,.
\]
\[
(\Gamma(\tau)-\limsup_{h\to\infty} F_h)(x) = \sup_{U\in\mathcal U(x)}\limsup_{h\to\infty}\inf_{U}F_h\,.
\]
They are called, respectively, the {\it $\Gamma$-lower limit} and {\it $\Gamma$-upper limit }of the sequence
$(F_h)_h$ in the topology $\tau$. 

Then, we give the following definition.
\begin{defi} Let $(F_h)_h$ and $F$ be functionals from space $(X,\tau)$ to $\bar\bbR$.
%\begin{itemize}
%\item[(i)]
%We say that {\it $(F_h)_h$ sequentially $\Gamma$-converges to $F$ at $x$ } in relation to the topology $\tau$ if
%\[
%s -\Gamma(\tau) \liminf_{h\to\infty} F_h(x) =\, s -\Gamma(\tau) \limsup_{h\to\infty} F_h(x) =\, F (x)
% \]
%and we write
%\[
%F(x) = s-\Gamma(\tau ) \lim_{h\to\infty}F_h(x)\,.
%\]
%\item[(ii)] 
We say that {\it $(F_h)_h$  $\Gamma(\tau)$-converges to $F$}, or also that {\it $(F_h)_h$ $\Gamma$-converges to $F$ in the topology $\tau$, at $x\in X$, }if
\[
(\Gamma(\tau)-\liminf_{h\to\infty} F_h)(x) =\, (\Gamma(\tau)-\limsup_{h\to\infty} F_h)(x) =\, F (x)
 \]
and we write
\[
F(x)=(\Gamma(\tau )-\lim_{h\to\infty}F_h)(x)\,.
\]
%\end{itemize}
\end{defi}
Let us recall below some relevant properties concerning $\Gamma$-convergence that we will need later.

\begin{teo}\label{properGammaconv}Let $F_h$ and $F$ be functionals from space $(X,\tau)$ to $\bar\bbR$.
\begin{itemize}
\item[(i)]$($\cite[Proposition 6.1]{DM}$)$ If $(F_h)_h$  $\Gamma(\tau)$-converges to $F$, then each of its subsequence $(F_{h_k})_k$ still $\Gamma(\tau)$-converges to F.

\item[(ii)]$($\cite[Proposition 6.3]{DM}$)$ Let $\tau_i$, $i=1,2$, be two topologies on $X$ and suppose that $\tau_1$ is weaker than $\tau_2$. If $(F_h)_h $ $\Gamma(\tau_1)$-converges to $F_1$ and $\Gamma(\tau_2)$-converges to $F_2$, then $F_1\le\, F_2$.
\item[(iii)]$($\cite[Theorem 7.8]{DM}$)$ $($Fundamental Theorem of $\Gamma$-convergence$)$ Assume that the sequence $(F_h)_h$ is equicoercive $($on X$)$, that is, for each $t\in\bbR$ there exists a closed countably compact $K_t\subset X$ such that 
 \[
 \{x\in X:\,F_h(x)\le\,t\}\subset K_t\quad\text{ for each }h\,.
 \]
Let us also assume that $(F_h)_h$ $\Gamma(\tau)$-converges to $F$. Then $F$ is coercive and
\[
\min_{x\in X}F(x)=\,\lim_{h\to\infty}\inf_{x\in X}F_h(x)\,.
\]
\item[(iv)]$($\cite[Proposition 8.1]{DM}$)$ Assume that $(X,\tau)$ satisfies the first countability axiom. Then $(F_h)_h$ $\Gamma(\tau)$-converges to $F$ if and only if the following two conditions hold:
\begin{itemize}
\item[(1)] $(${\rm $\Gamma-\liminf$ inequality}$)$ for any $x\in X$ and for any sequence $(x_h)_h$ converging to $x$ in $X$ one has
\[
F(x)\le \liminf_{h\to\infty} F_h(x_h)\,;
\]
\item[(2)] $(${\rm $\Gamma-\lim$ equality}$)$ for any $x\in X$, there exists a sequence $(x_h)_h$ converging to $x$ in $X$ such that 
\[
F(x)=\lim_{h\to\infty} F_h(x_h)\,.
\]
\end{itemize}
\item[(v)]$($\cite[Theorem 8.5]{DM}$)$ Assume that $(X,\tau)$ satisfies the second countability axiom, that is, there is a countable base for the topology $\tau$. Then every sequence $(F_h)_h$ of functionals from $X$ to $\bar\bbR$ has a $\Gamma(\tau)$-convergent subsequence.
\end{itemize}
\end{teo}
\begin{oss} It is well-known that inequality in Theorem \ref{properGammaconv} (ii) can be strict, even in the case of a (infinite dimensional) Banach space $X$, $\tau_1\equiv$ weak topology of $X$ and $\tau_2\equiv$ strong topology of $X$ (see, for instance, \cite[Example 6.6]{DM}).  An instance of such a phenomenon  can occur in the case of non-coercive quadratic integral functionals \cite{ACM}.
\end{oss}

%%%%%%%%
\begin{defi}[$\bar\Gamma$-convergence for local functional on $L^p(\Om)\times\mA$]
%\begin{itemize}
%\item[(i)] Let $F_h:\,L^p(\Om)\to [0,\infty]$ ($h=1,2,\dots$)  and $F:\,L^p(\Om)\to [0,\infty]$. We say that the sequence $(F_h)_h$ $\Gamma$- converges to $F$, and we will write $F=\,\Gamma-\lim_{h\to\infty} F_h $,  if the following conditions are satisfied:
%\subitem[$\Gamma-\liminf$ inequality] for each $u\in L^p(\Om)$ and $(u_h)_h\subset L^p(\Omega)$  converging to $u$ in $L^p(\Om)$ ,
%\begin{equation*}\label{Gammaliminf}
%F(u)\le\,\liminf_{h\to\infty}F_h(u_h)\,;
%\end{equation*}
%\subitem[$\Gamma-\limsup$ inequality] for each $u\in L^p(\Om)$ there exists $(u_h)_h\subset L^p(\Omega)$  converging to $u$ in $L^p(\Om)$ with
%\begin{equation*}\label{Gammalimsup}
%F(u)\ge\,\limsup_{h\to\infty}F_h(u_h)\,.
%\end{equation*}
  Let $F_h:\,L^p(\Om)\times\mA\to [0,\infty]$ ($h=1,2,\dots$) be a sequence of increasing functionals. We say that the sequence $(F_h)_h$ $\bar\Gamma$-converges to a functional $F:\,L^p(\Om)\times\mA\to [0,\infty]$, and we will write $F=\,\bar\Gamma-\lim_{h\to\infty}F_h$, if $F$ is increasing, inner regular and lsc and the following conditions are satisfied:
\subitem[$\bar\Gamma-\liminf$ inequality] for each $u\in L^p(\Om)$, for every $A\in\mA$ and $(u_h)_h\subset L^p(\Omega)$  converging to $u$ in $L^p(\Om)$, it holds
\begin{equation*}\label{barGammaliminf}
F(u,A)\le\,\liminf_{h\to\infty}F_h(u_h,A)\,;
\end{equation*}
\subitem[$\bar\Gamma-\limsup$ inequality] for each $u\in L^p(\Om)$,  for each $A,\,B\in\mA$ with $A\Subset B$, there exists $(u_h)_h\subset L^p(\Omega)$  converging to $u$ in $L^p(\Om)$ with
\begin{equation*}\label{barGammalimsup}
F(u,B)\ge\,\limsup_{h\to\infty}F_h(u_h,A)\,.
\end{equation*}
%\end{itemize}
\end{defi}
\begin{oss}\label{GammalimimpliesbarGammalim}Let us consider a sequence of increasing functionals $F_h:\,L^p(\Om)\times\mA\to [0,\infty]$ ($h=1,2,\dots$). Assume that there exists   a measure functional $F:\,L^p(\Om)\times\mA\to [0,\infty]$ such that $(F_h(\cdot,A))_h$ $\Gamma$-converges to $F(\cdot,A)$ for each $A\in\mA$ . Then $(F_h)_h$ $\bar\Gamma$-converges to $F$. Indeed, being $F$ a $\Gamma$-limit, it is lsc (see \cite[Propostion 6.8]{DM}) and it is increasing and inner regular, because it is a measure. Moreover the $\bar\Gamma-\liminf$ and $\bar\Gamma-\limsup$ inequalities immediately follows by the characterization of $\Gamma$-limit in  Theorem \ref{properGammaconv} (iv).
\end{oss}

\begin{defi}  Let $F:\, L^p(\Om)\times\mA\to [0,\infty]$ be a non-negative functional.
We say that $F$ satisfies the {\it fundamental estimate} if, for every $\eps>\,0$ and for
every $A', \, A'',\, B \in\mA$, with $A' \Subset A''$, there exists a constant $M >\, 0$ with
the following property: for every $u,\,v\in L^p(\Om)$, there exists a function
$\varphi\in\Co\infty_c(A'')$, with $0\le\vf\le1$ on $A''$, $\vf=1$ in a neighborhood of $A'$ , such that
\[
\begin{split}
F\Big(\vf u&+(1-\vf)v, A'\cup B\Big)\le\,(1+\eps)\Big(F(u,A'')+F(v,B)\Big)+\\
&+\eps\Big( \norma u_{L^p(S)}^p+\norma v_{L^p(S)}^p+1\Big)+ M \norma {u-v}_{L^p(S)}\,,
\end{split}
\]
where $S = (A''\setminus A')\cap B$. Moreover, if $\F$ is a class of non-negative functional
on $L^p(\Om)\times\mA$, we say that the {\it fundamental estimate holds uniformly in $\F$} if
each element $F$ of $\F$ satisfies the fundamental estimate with $M$ depending
only on $\eps$, $A'$, $A''$, $B$ while $\vf$ may depend also on $F,\,u, \,v$.
\end{defi}

\begin{oss} Let us recall  that, if $F=\,\bar\Gamma-\lim_{h\to\infty}F_h$
and $F_h:\,L^p(\Om)\times\mA\to[0,\infty]$ are {\it measures}, then $F$ need not be a measure (see \cite[Examples 16.13 and 16.14]{DM}). If the sequence $(F_h)_h$ satisfies the fundamental estimates uniformly with respect to $h$, then $F$ is a measure (see \cite[Theorem 18.5]{DM}).
\end{oss}

Let us now state a result which assures the coincidence between the $\bar\Gamma-\lim F_h$ and $\Gamma-\lim F_h$ for a sequence of local functional $F_h:\,L^p(\Om)\times\mA\to [0,\infty]$, provided that the 
fundamental estimate holds uniformly for the sequence $(F_h)_h$ \cite[Theorem 18.7]{DM}.

\begin{teo}\label{DMThm187} Let $(F_h)_h$ be a sequence of non-negative increasing functionals
 on $L^p(\Om)\times\mA$ which $\bar\Gamma$ -converges to a functional $F$. Assume that there
exist two constants $c_1\ge\,1$ and $c_2\ge\,0$, a non-negative increasing functional
$G:\,L^p(\Om)\times\mA\to [0,\infty]$, and a non-negative Radon measure $\mu: \mathcal B(\Om)\to [0,\infty]$
such that
\begin{equation*}
G(u,A)\le\,F_h(u,A)\le\,c_1\,G(u,A)+\,c_2\norma u_{L^p(A)}^p+\,\mu(A)
\end{equation*}
for every $u\in L^p(\Om)$, $A\in\mA$ and $h\in\bbN$. Assume, in addition, that $G$ is a lower
semicontinuous measure and that the fundamental estimate holds uniformly
for the sequence $(F_h)_h$. Then, $(F_h(\cdot,A))_h$ $\Gamma$-converges in $L^p(\Om)$ to $F(\cdot,A)$ for
every $A\in\mA$ such that $\mu(A)<\,\infty$.
\end{teo}

\subsection{Convergence of integrands and $\Gamma$-convergence for integral functionals depending on vector fields}
In this section we will deal with integral functionals $F: W^{1,p}_X(\Om)\to\bbR$, with $\Omega$  bounded open subset of $\Rn$ and $p>\,1$, of the form
\begin{equation}\label{FonWX}
F(u):=\,\int_\Om f(x,Xu)\,dx
\end{equation}
where the integrand $f:\,\Om\times\Rm\to\bbR$ belongs to class $I_{m,p}(\Omega,c_0,c_1)$ (i.e., $f$ satisfies $(I_1)$, $(I_2)$ and $(I_3)$ in the Introduction).

% is a function verifying the following assumptions:
%\begin{itemize}
%\item[($I_1$)] for every $\eta\in\Rm$, the function $f(\cdot,\eta):\,\Om\to [0,\infty]$ is Borel measurable on $\Om$;
%\item[($I_2$)] for a.e. $x\in\Om$, the function $f(x, \cdot):\,\Rm\to [0,\infty)$ is convex;
%\item[($I_3$)] there exists constants $c_1>\,c_0\ge\,0$ such that
%\[
%c_0\,|\eta|^p\le\,f(x,\eta)\le\,c_1\left(\left|\eta\right|^p+1\right)\,,
%\]
%for a.e. $x\in\Om$ and for each $\eta\in\Rm$.
%\end{itemize}
%We will denote in the following by $I_{m,p}(\Omega,c_0,c_1)$ the class of such integrand functions.\\
It is easy to show, taking \cite[Proposition 5.12]{DM} into account, that the following $\Gamma$-convergence results still hold.
\begin{prop}\label{convpointimplgamma} Let $(f_h)_h$ and $f$ be functions in $I_{m,p}(\Om,0,c_1)$. Let $F_h,\,F:\,W^{1,p}_X(\Om)\to\bbR$ be  the corresponding integral functionals in \eqref{FonWX}. Assume  that
\begin{equation}
F_h\to F\ \text{$($pointwise$)$ in}\ W^{1.p}_X(\Om)\,.
\end{equation}
Then $(F_h)_h$ $\Gamma$-converges to $F$ in $W^{1.p}_X(\Om)$, i.e.,
\begin{equation}
F(u)=\,(\Gamma(W^{1,p}_X(\Om))-\lim_{h\to\infty}F_h)(u)\quad\forall\,u\in W^{1,p}_X(\Om)\,.
\end{equation}
\end{prop}
The following theorem, in particular, shows that the pointwise convergence of the integrands also implies the $\Gamma$-convergence of the corresponding integral functionals in the weak topology of $W^{1,p}_X(\Om)$.
\begin{teo}Let $(f_h)_h$ and $f$ be functions in $I_{m,p}(\Om,0,c_1)$. Let $F_h,\,F:\,W^{1,p}_X(\Om)\to\bbR$ be the corresponding integral functionals in \eqref{FonWX}. Assume that
\begin{equation}\label{pointwiseconvintegr}
f_h(\cdot,\eta)\to f(\cdot,\eta)\text{ a.e. in }\Omega,\,\text{for each }\eta\in\Rm\,.
\end{equation}
Then
\begin{equation}
F(u)=\,(\Gamma(W^{1,p}_X(\Om){\rm -weak})-\lim_{h\to\infty}F_h)(u)\quad\forall\,u\in W^{1,p}_X(\Om)\,,
\end{equation}
i.e., $(F_h)_h$ $\Gamma$-converges
to $F$ in the weak topology of $W^{1,p}_X(\Om)$.
\end{teo}
The scheme of the proof trivially follows the one of  \cite[Theorem 5.14]{DM} and we omit it.
\subsection{ $\Gamma$-compactness results for  integral functional depending on vector fields } The main result of this section is the following.

\begin{teo}\label{mainthm} Let $\Om\subset\Rn$ be a bounded open set and let $X=(X_1,\dots,X_m)$   satisfy (LIC) on $\Om$. Let $(f_h)_h\subset I_{m,p}(\Om,c_0,c_1)$ and, for each $h$, let $F^*_h:\,L^p(\Om)\times\mA\to [0,\infty]$ be the local functional defined as 

\begin{equation}\label{Fstarh}
F^*_h(u,A):=
\displaystyle{\begin{cases}
\int_{A}f_h(x,Xu(x))dx&\text{ if }A\in\mA,\,u\in W^{1,p}_X(A)\\
\infty&\text{ otherwise}
\end{cases}
\,.
}
\end{equation}

Then, up to a subsequence, there exist a local functional $F:\,L^p(\Om)\times\mA\to [0,\infty]$ and $f\in I_{m,p}(\Om,c_0,c_1)$ such that 
\begin{itemize}
\item[(i)] \eqref{GammalimitFeuc} holds;
\item[(ii)] $F$ admits the following representation 
\begin{equation}\label{GammalimitF}
F(u,A):=
\displaystyle{\begin{cases}
\int_{A}f(x,Xu(x))dx&\text{ if }A\in\mA,u\in W^{1,p}_X(A)\\
\infty&\text{ otherwise}
\end{cases}
\,.
}
\end{equation}

\end{itemize}
\end{teo}

Let us begin to recall a fundamental result about the representantion of the $\bar\Gamma$-limit with respect to a Euclidean integrand \cite[Theorem 20.3] {DM}, which applies to a large class of integral functionals. Let $c_1,\,c_2,\,c_3$ be real numbers with $c_i\ge 0$ $i=1,2,3$. Let us denote by $\mathcal H=\mathcal H(p,c_1,c_2, c_3)$ the class of all local functionals $F:\,L^p(\Omega)\times\mA\to [0,\infty]$ for
which there exist two Borel functions $f_e,\,g:\,\Om\times\Rn\to [0,\infty)$ (depending on $F$) such that
\begin{itemize}
\item[(a)] $
F(u,A):=
\displaystyle{\begin{cases}
\int_{A}f_e(x,Du)dx&\text{ if }A\in\mA,u\in W^{1,1}_{\rm loc}(A)\\
\infty&\text{ otherwise}
\end{cases}
\,
}
$
;
\item[(b)] $g(x,\xi)\le\,f_e(x,\xi)\le\, c_1\left(g(x,\xi)+1\right)$;
\item[(c)] $0\le\,g(x,\xi)\le\, c_2\left(|\xi|^p+1\right)$;
\item[(d)] $g(x,\cdot)$ is convex on $\Rn$;
\item[(e)] $g(x,2\xi)\le\, c_3\left(g(x,\xi)+1\right)$,
\end{itemize}
for every $u\in L^p(\Om)$, $x\in\Om$, $\xi\in\Rn$.

\begin{teo}\label{DMThm203} For every  sequence $(F_h)_h$ of functionals of the class $\mathcal H$
there exist a subsequence $(F_{h_k})_k$ and an increasing functional $F:\,L^p(\Omega)\times\mA\to [0,\infty]$ such that $(F_{h_k})_k$ $\bar\Gamma$-converges to $F$. The functional $F$ can be represented in integral form by a Euclidean integrand, that is, there exists a Borel function
$f_e:\,\Omega\times\Rn\to [0,\infty]$ verifying
\begin{itemize}
\item[(i)] $f_e(x,\cdot)$ is  convex on $\Rn$;
\item[(ii)] $0\le\,f_e(x,\xi)\le\,c_1(c_2+1)+\,c_1\,c_2|\xi|^p$ for a.e. $x\in\Om$, for each $\xi\in\Rn$,
\end{itemize}
such that \eqref{GammalimitFeucrepresW1p} holds.
%\begin{equation}\label{GammalimitFeucrepresW1p}
%F(u,A):=\int_{A}f_e(x,Du(x))\,dx
%\end{equation}
%for every $A\in\mA$, for every $u\in L^p(\Om)$ such that $u|_A\in W ^{1,p}(A)$
\end{teo}

Let us also recall an useful criterion for proving that a class of local functionals on $L^p(\Om)\times\mA$ satisfies the fundamental estimate uniformly \cite[Theorem 19.4]{DM} and a $\bar\Gamma$-compacness result in this class  \cite[Theorem 19.5]{DM}. 

\begin{teo}\label{DMThm194}Let $c_i$ ($i=1,2,3,4$) be non negative real numbers and let $\sigma:\,\mA\to [0,\infty]$ be a superadditive  increasing set function such that $\sigma(A)<\,\infty$ for each $A\Subset\Om$. Let $\mathcal F'=\,\mathcal F'(p, c_1,c_2,c_3,c_4)$ be the class of all non-negative increasing local  
functionals $F: \, L^p(\Om)\times\mA\to [0,\infty]$ with the following properties: $F$ is a measure and
there exists a non-negative increasing local functional $G: \, L^p(\Om)\times\mA\to [0,\infty]$
(depending on $F$) such that $G$ is a measure and
\begin{equation}\label{DM197}
G(u,A)\le\,F(u,A)\le\,c_1G(u,A)+\,c_2\norma u_{L^p(A)}^p+ \sigma(A)\,;
\end{equation}
\begin{equation}\label{DM198}
\begin{split}
&G(\vf u+(1-\vf) v,A)\le\,c_4\left(G(u,A)+\,G(v,A)\right)+\\
&+c_3c_4\max_{\Om}|D\vf|^p\norma {u-v}_{L^p(A)}^p+\,2c_4\sigma(A)\,,
\end{split}
\end{equation}
for every $u,\,v\in L^p(\Om)$, $A\in\mA$, $\vf\in\Co\infty_c(\Om)$ with $0\le\,\vf\le\,1$. Then, the fundamental estimate holds uniformly on $\mathcal F'$.
\end{teo}

\begin{teo}\label{DMThm195} Let $\mathcal F'=\,\mathcal F'(p, c_1,c_2,c_3,c_4)$ be the class of local  
functionals defined in Theorem \ref{DMThm194}. For every sequence $(F_h)_h \subset\mathcal F'$, there exists a subsequence $(F_{h_k})_k$ which $\bar\Gamma$-converges to a lower
semicontinuous functional $F \in\mathcal F'$.
\end{teo}

Let us now introduce some results concerning functionals depending on vector fields.  
Let us first prove  a $\Gamma$-compactness result (see Theorem \ref{DMThm196ext}) for a class of local functional on $L^p(\Omega)\times\mA$ satisfying suitable growth conditions with respect to the local functional $\Psi_p:\,L^p(\Om)\times\mA\to [0,\infty]$ defined as
\begin{equation}\label{Psip}
\Psi_p(u,A):=\displaystyle{\begin{cases}
\int_{A}|Xu|^p\,dx&\text{ if }A\in\mA,u\in W^{1,p}_X(A)\\
+\infty&\text{ otherwise}
\end{cases}
\,.
}
\end{equation}
As a consequence, we will get a $\Gamma$-compactness result for a class of integral functionals represented with respect to Euclidean integrands, but still with growth condition with respect to to  $\Psi_p$ (see Theorem \ref{DMThm204ext}). The former  is an extension of \cite[Theorem 19.6]{DM}, the latter of \cite[Theorem 20.4]{DM}.

\begin{lem}\label{Psipprop} Let $p>\,1$. Then $\Psi_p:\,L^p(\Om)\times\mA\to [0,\infty]$  is  a measure and lsc.
\end{lem}
\begin{proof}
Let us start by proving that for any $A\in \mA$ the function $u\to \Psi_p(u,A)$ is $L^p-$lsc, i.e., for any $A\in \mA$ and $(u_h)_h\subset L^p(\Om)$, $u_h\to u$ in $L^p(\Om)$, it satisfies
\begin{align}
\Psi_p(u,A)\leq \liminf_{h\to \infty}\Psi_p(u_h,A).
\end{align}
We can assume $\liminf_{h\to \infty}\Psi_p(u_h,A)<\infty$. Therefore, up to a subsequence, we can also assume that $\lim_{h\to\infty}\Psi_p(u_h,A)$ exists. Hence $(u_h)_h$ is bounded in $W^{1,p}_X(A)$ and, since  $W^{1,p}_X(A)$ is reflexive (recall Proposition \ref{psaw} and that $p>1$), we get a subsequence $u_h\rightharpoonup u$ in $W^{1,p}_X(A)$ and, in particular, $Xu_h\rightharpoonup Xu$ in $L^p(A)$, which implies the conclusion, recalling the lower semicontinuity of the $L^p-$norm with respect to the weak convergence.

We now prove that for any $u\in L^p(\Om)$ the function $\Psi_p(u,\cdot):\mA\to [0,\infty]$ is a measure, i.e., there exists a Borel measure $\mu_u:\mathcal{B}(\Om)\to [0,\infty]$ such that $\Psi_p(u,A)=\mu_u(A)$ for every $A\in \mA$. Since, by Remark \ref{charmeas}, $\Psi_p(u,\cdot)$ is nonnegative, increasing and such that $\Psi_p(u,\emptyset)=0$, it suffices to prove that $\Psi_p(u,\cdot)$ is subadditive, superadditive and inner regular.

$\Psi_p(u,\cdot)$ is subadditive, namely for every  $A,A_1,A_2\in \mathcal{A}$ with $ A\subseteq A_1\cup A_2$ 
\begin{align}\label{subadditive}
\Psi_p(u,A)\leq \Psi_p(u,A_1)+\Psi_p(u,A_2).
\end{align}
We can assume $u\in W^{1,p}_X(A_1)\cap W^{1,p}_X(A_2)$ and $A_1,A_2\in \mathcal{A}$, otherwise the conclusion is trivial. Remark \ref{extpropoW1pXloc} (ii) gives $u\in W^{1,p}_X(A_1\cup A_2)$, therefore $\Psi_p(u,A_1\cup A_2)=\int_{A_1\cup A_2} |Xu|^p\, dx$ and \eqref{subadditive} follows.

$\Psi_p(u,\cdot)$ is superadditive, namely for every  $A,A_1,A_2\in \mathcal{A}$ with $ A_1\cup A_2\subseteq A$ and $A_1\cap A_2=\emptyset$
\begin{align}\label{peradditive}
\Psi_p(u,A)\geq \Psi_p(u,A_1)+\Psi_p(u,A_2).
\end{align}
We can assume $u\in W^{1,p}_X(A)$ and $A\in \mathcal{A}$, otherwise the conclusion is trivial. Remark \ref{extpropoW1pXloc} (iv) gives 
$u\in W^{1,p}_X(B)$ for any open set $B\subseteq A$. Let $A,A_1,A_2\in \mathcal{A}$, $ A_1\cup A_2\subseteq A$ and $A_1\cap A_2=\emptyset$. Then 
\[
\Psi_p(u,A_1)+\Psi_p(u,A_2)=\int_{A_1\cup A_2} |Xu|^p\, dx\leq \int_{A} |Xu|^p\, dx
\]
and \eqref{peradditive} follows.

$\Psi_p(u,\cdot)$ is inner regular, namely for every $A\in \mathcal{A}$
\begin{align}
\Psi_p(u,A)=\sup\left\{\Psi_p(u,B)\ |\ B\in\mathcal{A},\ B\Subset A\right\}.
\end{align}
Let $M:=\sup\left\{\Psi_p(u,B)\ |\ B\in\mathcal{A},\ B\Subset A\right\}\in [0,+\infty]$.
If $M=+\infty$, there exists $\{B_i\}_{i\in\mathbb{N}}\subset \mathcal{A}$, $B_i\Subset A$ such that $\Psi_p(u,B_i)\to \infty$ as $i\to +\infty$ and the conclusion follows by observing that for all $i\in\mathbb{N}$, $\Psi_p(u,B_i)\leq \Psi_p(u,A)$.
If $M\in [0,\infty)$, then $\|u\|_{W^{1,p}_X(B)}\leq M$ for any $B\in \mathcal{A}$, $B\Subset A$. Then, Remark \eqref{extpropoW1pXloc} (iii) gives $u\in W^{1,p}_X(A)$ and, by definition, $\Psi_p(u,A)=\int_A |Xu|^p\, dx$. For any $\varepsilon>0$ there exists $\delta>0$ such that $\int_E |Xu|^p\, dx\leq \varepsilon$ for any $E\in \mathcal{A}$ with $|E|\leq \delta$. Let $B\Subset A$ such that $|A\setminus \overline{B}|\leq \delta$, then
\[
\int_A |Xu|^p\, dx=\int_B |Xu|\, dx+\int_{A\setminus \overline{B}}|Xu|^p\, dx\leq \int_B |Xu|^p\, dx+ \varepsilon
\]
and the thesis follows.
\end{proof}
\begin{teo}\label{DMThm196ext} Let $p>\,1$, $\Om\subset\Rn$ be a bounded open set and $c_1\ge\,c_0>\,0$. Denote by $\mathcal M=\,\mathcal M(p,c_0,c_1)$
the class of local functionals $F: \,L^p(\Om)\times\mA\to [0,\infty]$ such that  $F$ is a measure
and
\begin{equation}\label{controlFbyPsip}
c_0\,\Psi_p(u,A)\le\,F(u,A)\le\,c_1\left(\Psi_p(u,A)+\norma u_{L^p(A)}^p+\,|A|\right)
\end{equation}
for every $u\in L^p(\Om)$ and for every $A\in\mA$. Then, the fundamental  
estimate holds uniformly in $\mathcal M$ and every sequence $(F_h)_h\subset\mathcal M$ has a  
subsequence $(F_{h_k})_k$ which $\bar\Gamma$-converges to a functional $F$ of the class $\mathcal M$.  
Moreover, $(F_{h_k}(\cdot,A))_k$ $\Gamma$-converges to $F(\cdot,A)$ in $L^p(\Om)$ and
\begin{equation}\label{domGammalimit}
\mathrm{dom}F(\cdot,A):=\left\{u\in L^p(\Om):\,F(u,A)<\,\infty\right\}=\W 1p(A)\,
\end{equation}
for every $A\in\mA$. 
\end{teo}
\begin{proof} Let us begin to prove that the fundamental estimate holds uniformly in $\mathcal M$ .
Let
\begin{equation}\label{gdef}
g(x,\xi):=\,c_0\,|C(x)\xi|^p\quad\text{ if }x\in\Om,\,\xi\in\Rn\,.
\end{equation}
Notice that, since the entries of matrix $C(x)$ are  Lipschitz continuous functions,
\begin{equation}\label{propg3}
g(x,\xi)\le\,c_0\sup_{\Om}\|C(x)\|^p\,|\xi|^p=\, c_2|\xi|^p\quad\text{ if }x\in\Om,\,\xi\in\Rn\,,
\end{equation}

\begin{equation}\label{propg1}
g(x,2\xi)=\,2^{p-1}\,2g(x,\xi)=\,c_3\,2g(x,\xi)\quad\text{ if }x\in\Om,\,\xi\in\Rn\,
\end{equation}
and
\begin{equation}\label{propg2}
g(x,\cdot)\text{ is convex on } \Rn\,.
\end{equation}
Thus, from \eqref{propg3}, \eqref{propg1} and \eqref{propg2}, arguing as in \cite[(19.6)]{DM}, it follows that
\begin{equation}\label{estg}
g(x,t\xi+(1-t)\eta+\zeta)\le\,c_3\left(g(x,\xi)+g(x,\eta)\right)+\,c_2\,|\zeta|^p
\end{equation}
for every $x\in\Om$, $t\in [0,1]$, $\xi,\,\eta\in\Rn$. We are going to apply Theorem \ref{DMThm194}. Observe that, choosing $G=\,c_0\,\Psi_p$,  from \eqref{controlFbyPsip}, \eqref{DM197} immediately holds with
 \[
 c_1\equiv \frac{c_1}{c_0},\;c_2\equiv c_1,\;\sigma(A)=\,c_1\,|A|\,.
 \]
 Let us show \eqref{DM198}. By \eqref{estg}, it follows that
 \[
 \begin{split}
 &G\left(\vf u+(1-\vf)v,A\right)=\,\int_A g\left(x,\vf Du+(1-\vf)Dv+(u-v) D\vf\right)dx\\
 &\le\,\int_A \left[c_3\left(g(x,Du)+g(x,Dv\right)+c_2 |D\vf|^p |u-v|^p\right]dx\\
 &\le\,c_3\left(G(u,A)+G(v,A)\right)+\,c_2\left(\max_{\Om}|D\vf|^p\right)\,\norma{u-v}_{L^p(A)}^p\,
\end{split}
 \]
 for each $u,\,v\in L^p(\Om)$, $A\in\mA$, $\vf\in\Co\infty_c(\Om)$ with $0\le\,\vf\le\,1$.
 Thus \eqref{DM198} holds with
 \[
 c_4\equiv c_3\text{ and }c_3c_4\equiv c_2.
 \]
Thus we get the desired conclusion. From Theorem \ref{DMThm195}, every sequence $(F_h)_h\subset\mathcal M$ has a subsequence  $(F_{h_k})_k$ $\bar\Gamma$-converging to  a functional $F:\,L^p(\Om)\times\mA\to [0,\infty]$ which is a measure. As each functional $F_h$  satisfies \eqref{controlFbyPsip}, the functional $F$ satisfies \eqref{controlFbyPsip}, since $\Psi_p$ is lsc and inner regular by Lemma \ref{Psipprop} and Remark \ref{charmeas}. By applying Theorem \ref{DMThm187}, we get that $(F_{h_k}(\cdot,A))_k$ $\Gamma$-converges to $F(\cdot,A)$ in $L^p(\Om)$ for each $A\in\mA$, since $\Om$ is bounded. Finally, by \eqref{controlFbyPsip}, \eqref{domGammalimit} follows.
\end{proof}
Let $p>\,1$ and let $c_1\ge\,c_0$, let $\Om\subset\Rn$ be a bounded open set. Let us denote by $\mathcal I=\,\mathcal I(p,c_0.c_1)$ the class of local functionals $F:\,L^p(\Om)\times\mA\to [0,\infty]$ for which there exists a Borel function $f_e:\,\Om\times\Rn\to [0,\infty)$ such that
\begin{itemize}
\item[(i)] claim (a) of properties defining $\mathcal H$ holds;
\item[(ii)] $c_0\,|C(x)\xi|^p\le\,f_e(x,\xi)\le\,c_1\left(\left|C(x)\xi\right|^p+1\right)$ a.e. $x\in\Om$, for each $\xi\in\Rn$.
\end{itemize}
\begin{teo}\label{DMThm204ext} For every sequence $(F_h)_h\subset\mathcal I$ there exist a subsequence $(F_{h_k})_k$ and a measure functional $F:\,L^p(\Om)\times\mA\to [0,\infty]$ such that $(F_{h_k}(\cdot,A))_k$  $\Gamma$-converges to $F(\cdot,A)$ in $L^p(\Om)$ and \eqref{domGammalimit} holds for every $A\in\mA$. Moreover  there exists a Borel function $f_e:\,\Om\times\Rn\to [0,\infty)$, convex in the second variable and satisfying (ii), for which \eqref{GammalimitFeucrepresW1p} holds.
\end{teo}
\begin{proof} By Theorem \ref{DMThm196ext}, for each $(F_h)_h\subset\mathcal I$ there exist a subsequence $(F_{h_k})_k$ and  an inner regular functional $F:\,L^p(\Om)\times\mA\to [0,\infty]$ such that $(F_{h_k}(\cdot,A))_k$  $\Gamma$-converges to $F(\cdot,A)$ in $L^p(\Om)$ for every $A\in\mA$. Moreover, since $\Psi_p$ is lsc and inner regular, for each $u\in L^p(\Om)$, $A\in\mA$,
\begin{equation}\label{controlFbyPsip2}
c_0\,\Psi_p(u,A)\le\,F(u,A)\le\,c_1\left(\Psi_p(u,A)+\,|A|\right)\
\end{equation}
where $\Psi_p$ is the local functional in \eqref{Psip}. If $g(x,\xi)$ is as in \eqref{gdef}, $\mathcal I(p,c_0.c_1)\subset\mathcal H(p,c'_1,c'_2, c'_3)$, for suitable $c'_i$ ($i=1,2,3$). From Theorem \ref{DMThm203} , there exists a Borel function $f_e:\,\Om\times\Rn\to [0,\infty)$, also convex in the second variable, for which \eqref{GammalimitFeucrepresW1p} holds.

 Let us now prove that  (ii) of properties defining  $\mathcal I$ holds. Let $u_\xi$ be  the function in \eqref{uxi}. From \eqref{controlFbyPsip2}, it follows that
\[
c_0\,\int_A|C(x)\xi|^p\,dx\le\,\int_A\Lef(x,\xi)\,dx\le\,c_1\,\left(|A|+\int_A|C(x)\xi|^p\,dx\right)
\]
for each $\xi\in\Rn$ and $A\in\mathcal A$. By means of the usual procedure, we can infer that there exists a negligeble set $\mathcal N\subset\Omega$ such that, for each $x\in\Omega\setminus\mathcal N$,
\[
c_0\,|C(x)\xi|^p\le\,f_e(x,\xi)\le\,c_1\left(\left|C(x)\xi\right|^p+1\right)\quad\forall\,\xi\in\bbQ^n\,.
\]
Then, since $\Lef(x,\cdot):\,\Rn\to [0,\infty)$ is continuous a.e. $x\in\Omega$, we can extend the previous inequality to all $\xi\in\Rn$.
\end{proof}
\begin{teo}\label{GammacompFheucl} Let $\Om\subset\Rn$ be a bounded open set, let $(f_h)_h\subset I_{m.p}(\Omega,c_0,c_1)$ and, for each $h$, let $F^*_h:\,L^p(\Om)\times\mA\to [0,\infty]$ be the local functional defined in \eqref{Fstarh}. Then, there exist a subsequence $(F^*_{h_k})_k$ and a measure functional $F:\,L^p(\Om)\times\mA\to [0,\infty]$ such that $(F^*_{h_k}(\cdot,A))_k$  $\Gamma$-converges to $F(\cdot,A)$ in $L^p(\Om)$ and \eqref{domGammalimit} holds for every $A\in\mA$. Moreover, there exists a Borel function $f_e:\,\Om\times\Rn\to [0,\infty)$, convex in the second variable, satisfying (ii) of properties defining $\mathcal I$, for which \eqref{GammalimitFeucrepresW1p} holds. 
\end{teo}
\begin{proof} Let  $(\Lefh)_h$ denote the sequence of Euclidean integrands in \eqref{Lefh} and let $(F_h)_h$ be the sequence of local functionals in \eqref{Fh}. Since $(\Lefh)_h\subset\mathcal I$, by applying Theorem \ref{DMThm204ext}, there exist a subsequence $(F_{h_k})_k$ and a measure functional $F:\,L^p(\Om)\times\mA\to [0,\infty]$ such that $(F_{h_k}(\cdot,A))_k$  $\Gamma$-converges to $F(\cdot,A)$ in $L^p(\Om)$ for every $A\in\mA$. Moreover, there exists a Borel function $f_e:\,\Om\times\Rn\to [0,\infty)$, convex in the second variable, satisfying (ii), for which \eqref{GammalimitFeucrepresW1p} holds.

By Theorem \ref{Theorem 3.1.1} (iii), it follows that, for each $h\in\bbN$, $A\in\mA$,
\begin{equation}\label{relaxedFh}
F^*_h(\cdot,A)=\,\bar F_h(\cdot,A)
\end{equation}
where $\bar F_h(\cdot,A):\,L^p(\Om)\to [0,\infty]$ denotes the relaxed functional of $F_h(\cdot,A):\,L^p(\Om)\to [0,\infty]$ with respect to the $L^p(\Om)$ topology (see \eqref{3.3}). By \eqref{relaxedFh} and a well-known property of $\Gamma$-convergence (see \cite[Propostion 6.11]{DM}), we also get  that $(F^*_{h_k}(\cdot,A))_k$  $\Gamma$-converges to $F(\cdot,A)$ in $L^p(\Om)$ for every $A\in\mA$.
\end{proof}

%Let us now carry out the second step of our proof strategy. The main result turns out to be the closure, w.r.t.  $\Gamma(L^p(\Om))$-convergence, of subclasses $J_1,\,J_2,\,J_3$ introduced above .

%\begin{teo}\label{GammaclosureJi}Let $\Om\subset\Rn$ be a bounded open set and let $X=(X_1,\dots,X_m)$   satisfy (LIC) on $\Om$. Let  $(f_h)_h\subset J_i$ ($i=1,2,3$) and, for each $h$, let $F^*_h:\,L^p(\Om)\times\mA\to [0,\infty]$ be the local functional defined in \eqref{Fstarh}.  Assume that:
%\begin{itemize}
%\item [(i)] there exists a measure functional $F:\,L^p(\Om)\times\mA\to [0,\infty]$ such that  $(F^*_h)_h$ $\Gamma$-converges to $F(\cdot,A)$ in $L^p(\Om)$ and \eqref{domGammalimit} holds  for each $A\in\mA$; 
%\item[(ii)]   there exists a Borel function $f_e:\,\Om\times\Rn\to [0,\infty)$, convex in the second variable, satisfying (ii), of properties defining $\mathcal I$, for which $F$ admits the integral representation  in \eqref{GammalimitFeucrepresW1p}.
%\item[(iii)] \eqref{domGammalimit} holds for every $A\in\mA$;
%%\item[(iv)] assume that,  in the case of subclasses $J_i$ with $i=1,2$, the family of vector fields $X=(X_1,\dots,X_m)$  simply satisfies (LIC) on $\Om$, while  in the case of subclass $J_3$, $X$ is a family of vector fields generating a Carnot group structure (see Definition \ref{Cargroupgen}).
%
%\end{itemize}
%
%Then there exists $f\in J_i$  for which $F$ admits integral representaion \eqref{GammalimitF}.
%\end{teo}

\begin{teo}\label{GammaclosureImp}Let $\Om\subset\Rn$ be a bounded open set and let $X=(X_1,\dots,X_m)$   satisfy (LIC) on $\Om$. Let  $(f_h)_h\subset I_{m,p}(\Om,c_0,c_1)$ and, for each $h$, let $F^*_h:\,L^p(\Om)\times\mA\to [0,\infty]$ be the local functional defined in \eqref{Fstarh}.  Assume that:
\begin{itemize}
\item [(i)] there exists a measure functional $F:\,L^p(\Om)\times\mA\to [0,\infty]$ such that  $(F^*_h)_h$ $\Gamma$-converges to $F(\cdot,A)$ in $L^p(\Om)$ and \eqref{domGammalimit} holds  for each $A\in\mA$; 
\item[(ii)]   there exists a Borel function $f_e:\,\Om\times\Rn\to [0,\infty)$, convex in the second variable, satisfying (ii) of properties defining $\mathcal I$, for which $F$ admits the integral representation  in \eqref{GammalimitFeucrepresW1p}.
\item[(iii)] \eqref{domGammalimit} holds for every $A\in\mA$.
%\item[(iv)] assume that,  in the case of subclasses $J_i$ with $i=1,2$, the family of vector fields $X=(X_1,\dots,X_m)$  simply satisfies (LIC) on $\Om$, while  in the case of subclass $J_3$, $X$ is a family of vector fields generating a Carnot group structure (see Definition \ref{Cargroupgen}).
\end{itemize}
Then, there exists $f\in I_{m,p}(\Om,c_0,c_1)$ for which $F$ admits the integral representation \eqref{GammalimitF}.
\end{teo}

\begin{proof}%[Proof of Theorem \ref{GammaclosureImp}]
Let us first notice that $\Lef$ satisfies  the assumptions of Lemma \ref{keylemma}. Thus we can assume that it satisfies \eqref{f0Vx}.

Let $f:\,\Om\times\Rm\to [0,\infty]$ be the function   in \eqref{deffbyfe}. Let us prove that $f\in I_{m,p}(\Om,c_0,c_1)$. Properties $(I_1)$ and $(I_2)$ follow from Therem \ref{reprvfeg}. Since $\Lef$ satisfies  (ii) of properties defining class $\mathcal I$, from \eqref{coincfef}, we can infer $(I_3)$.

From  Theorem \ref{reprvfeg} and Remark \ref{represFwrtXW1p}, $F$ admits the integral representation \eqref{GammalimitF}, but only for functions $u\in W^{1,p}(A)$. We are going to extend this representation to all functions $u\in W_X^{1,p}(A)$, by means of Theorem \ref{DMThm201ext} about the integral representation of local functionals with respect to $X$-gradient. Being $F$ a  $\Gamma$-limit, it is lsc (see \cite[Proposition 6.8]{DM}) and, by \cite[Proposition 16.15]{DM}, it is also local and, by assumptions, a measure. Thus assumptions (a), (b) and (c) of Theorem \ref{DMThm201ext} are satisfied. Let us prove assumtion (d). For every $h\in\bbN$, we have $F^*_h(u+c,A)=\, F^*_h(u,A)$ whenever $u\in L^p(\Omega)$, $c\in\bbR$. Then it is easy to see that this property also holds for the $\Gamma$-limit $F$. Let us now prove assumption (e). By the integral representation   \eqref{GammalimitFeucrepresW1p} and  Remark \ref{represFwrtXW1p}, it follows that, for each $A\in\mA$, $u\in W^{1,p}(A)$
\begin{equation}\label{conicintegrfef}
\begin{split}
F(u,A)&=\,\int_A\Lef(x,Du)\,dx=\,\int_Af(x,Xu)\,dx\\
&\le\,c_1\left(\int_A|Xu|^p+|A|\right)
\end{split}
\end{equation}
which implies property (e). 
%%%%%%%%%%%%%%%%%%%%%%%%%%%%%%%%%%%%%%%%%%%%%%%%%%%%%%%%%%%%%%%
%Eventually let us show property (f). By the integral representation   \eqref{GammalimitFeucrepresW1p}  it follows that, for each $x\in\Om$, $\xi\in\Rn$
%\[
%F(u_\xi,B(x,r) )=\,\int_{B(x,r)}\Lef (y,\xi)\,dy
%\]
%which implies that, for a.e. $x\in\Om$, there exists
%\[
%\lim_{r\to 0^+}\frac{F(u_\xi,B(x,r) )}{|B(x,r)|}=\,\Lef(x,\xi)\,.
%\]
%By \eqref{identcruc}, we get that assumption (f) is satisfied, too. 
%%%%%%%%%%%%%%%%%%%%%%%%%%%%%%%%%%%%%%%%
Thus there exists a Borel function $f^*:\,\Om\times\Rm\to [0,\infty]$ satisfying property (i) and (ii) of Theorem \ref{DMThm201ext}. In particular,  for each $A\in\mA$, $u\in\W1p(A)$
\[
F(u,A)=\int_A f^*(x,Xu)\,dx\,.
\]
By \eqref{conicintegrfef} and Theorem \ref{reprvfeg}, we get that $f(x,\eta)=\,f^*(x,\eta)$ for a.e. $x\in\Om$ and for each $\eta\in\Rm$. This concludes the proof.
\end{proof}
\begin{proof}[Proof of Theorem \ref{mainthm}] The proof immediately follows from Theorems \ref{GammacompFheucl} and \ref{GammaclosureImp}.
\end{proof}
We now introduce two integrand function subclasses $J_i\subset I_{m,p}(\Om,c_0,c_1)$ ($i=1,2$) for which the associated functionals in \eqref{Fstar} are still compact with respect to $\Gamma$- convergence in $L^p(\Om)$-topology. Let $\Omega\subset\Rn$ be  a bounded open set and let us fix $0<\,c_0\le\,c_1$.
\begin{itemize}
\item $J_1\equiv J_1(\Om,c_0,c_1)$ is the subclass of $I_{m,2}(\Omega,c_0,c_1)$ composed of integrand functions $f\in I_{m,2}(\Omega,c_0,c_1)$ which are quadratic forms with respect to $\eta$, that is,
\[
f(x,\eta)=\,\langle a(x)\eta,\eta\rangle=\,\sum_{i,j=1}^m a_{ij}(x)\eta_i\eta_j\quad\text{a.e. }x\in\Om, \forall\,\eta\in\Rm\,,
\]
with $a(x)=[a_{ij}(x)]$  $m\times m$ symmetric matrix . 

\item The subclass $J_2\equiv J_2(\Om,c_0,c_1)$ is composed by integrand functions $f\in I_{m,p}(\Omega,c_0,c_1)$ such that $f=\,f(\eta)$, that is, $f$ is independent of $x$.
\end{itemize}

\begin{teo}\label{subclasscomp} Let $\Om\subset\Rn$ be a bounded open set and let $X=(X_1,\dots,X_m)$   satisfy (LIC) on $\Om$. Let $(f_h)_h\subset J_i(\Om,c_0,c_1)$ ($i=1,2$) and, for each $h$, let $F^*_h:\,L^p(\Om)\times\mA\to [0,\infty]$ be the local functional defined in \eqref{Fstarh}.  Then, up to a subsequence, there exist a local functional $F:\,L^p(\Om)\times\mA\to [0,\infty]$ and $f\in J_i(\Om,c_0,c_1)$ such that 
\begin{itemize}
\item[(i)] \eqref{GammalimitFeuc} holds;
\item[(ii)] $F$ admits representation \eqref{GammalimitF}.
\end{itemize}
\end{teo}

\begin{proof}{\bf 1st case.}  Let us first show the conclusion for the subclass $J_1$.

Let $(f_h)_h\subset J_1$. By definition, we can assume that
\[
f_h(x,\eta):=\,\langle a_h(x)\eta,\eta\rangle\quad x\in\Om,\,\eta\in\Rm\,,
\]
where $a_h(x)=[a_{h,ij}(x)]$ is a $m\times m$ symmetric matrix  satisfying
\begin{equation}\label{ahcontrol}
c_0\,|\eta|^2\le\,\langle a_h(x)\eta,\eta\rangle\le\,c_1\left(\left|\eta\right|^2+1\right)\text{ a.e. }x\in\Om,\,\forall\eta\in\Rm\,
\end{equation}
\begin{equation}\label{ahLinfty}
a_{h,ij}\in L^\infty(\Om)\text{ for each }i,j=1,\dots,m,\,h\in\bbN\,.
\end{equation}

Applying Theorem \ref{mainthm}, up to a subsequence, there exist a local functional $F:\,L^p(\Om)\times\mA\to [0,\infty]$ and $f\in I_{m,2}(\Om,c_0,c_1)$ such that  \eqref{GammalimitFeuc} holds and $F$ admits representation \eqref{GammalimitF}. We have only to prove that 
\begin{equation}\label{finJ1}
f\in J_1\,.
\end{equation}

Notice that we can also assume that $F$ admits representation \eqref{GammalimitFeucrepresW1p} with
\[
\Lef(x,\xi):=\,f(x,C(x)\xi)\text{ for a.e. }x\in\Om\text{, for each }\xi\in\Rn\,.
\]
Moreover, by Theorem \ref{reprvfeg} (see \eqref{deffbyfe}  and \eqref{L-1repr}), it also holds the opposite representation, that is, for each $x\in\Om_X$,
\begin{equation}\label{coincfef0}
f(x,\eta)=\Lef(x,L_x^{-1}(\eta))\quad\forall\,\eta\in\Rm\,,
\end{equation}
with
\[
L_x^{-1}(\eta):=\,C(x)^T B(x)^{-1}\eta\,.
\]
Let us now consider the sequence of Euclidean integrands
\[
\begin{split}
f_{h,e}(x,\xi)&:=\,f_h(x,C(x)\xi)=\,\langle a_h(x)C(x)\xi,C(x)\xi\rangle\\
&=\,\langle C(x)^Ta_h(x)C(x)\xi,\xi\rangle=\,\langle a_{h,e}(x)\xi,\xi\rangle\,
\end{split}
\]
and the related local functionals $F_h:\,L^p(\Om)\times\mA\to [0,\infty]$ defined in \eqref{Fh}. Since $F_h(u,A)=\,F^*_h(u,A)$ for each $u\in W^{1,1}_{\rm loc}(A)$, by using well-known results of $\Gamma$-convergence for quadratic functionals (see  \cite[Theorem 22.1]{DM} and Remark \ref{GammalimimpliesbarGammalim}, one can easily prove that  there exists a $n\times n$ symmetric matrix $a_e(x)=\,[a_{e,ij}(x)]$, with $a_{e,ij}\in L^\infty(\Om)$ for each $i,j=1,\dots,n$ such that
\begin{equation*}
\Lef(x,\xi)=\,\langle a_e(x)\xi,\xi\rangle\text{ a.e. } x\in\Om,\,\forall\,\xi\in\Rn\,.
\end{equation*}
%,  for a.e. $x\in\Om$,
%\begin{equation}%\label{identcruc}
%\Lef(x,\xi_{N_x}+\zeta)=\,\Lef(x,\zeta)\quad\forall\,\xi,\,\zeta\in\Rn\,.
%\end{equation}
%Since $\Lef(x,\cdot)$ is a quadratic form,  it follows that
%\begin{equation}\label{propqf}
%\Lef(x,\xi_{N_x}+\zeta)=\,\Lef(x,\xi_{N_x})+\Lef(x,\zeta)+2\langle a_e(x)\xi_{N_x},\zeta\rangle\,
%\end{equation}
%and 
%\begin{equation}\label{CS}
%|\langle a_e(x)\xi_{N_x},\zeta\rangle|\le\,\sqrt{\Lef(x,\xi_{N_x})}\,\sqrt{\Lef(x,\zeta)}\,.
%\end{equation}
%By \eqref{eqfeonNx}, \eqref{propqf} and \eqref{CS},\eqref{identcruc} follows, Notice also that \eqref{identcruc} is equivalent to \eqref{f0Vx}, that is, for a.e. $x\in\Om$
%\[
%\Lef(x,\xi)=\,\Lef(x,\xi_{V_x})\quad\forall\,\xi\in\Rn\,.
%\]
% Let  $f:\,\Om\times\Rm\to [0,\infty]$ be the function   in \eqref{deffbyfe}. Then, by its definition and \eqref{L-1repr}, we get that, for each $x\in\Om_X$
By \eqref{coincfef0}, for each $x\in\Om_X$,
\begin{equation}\label{coincfef}
\begin{split}
f(x,\eta)&:=\Lef(x,L_x^{-1}(\eta))=\langle a_e(x)C(x)^T B(x)^{-1}\eta,C(x)^T B(x)^{-1}\eta\rangle\\
&=\,\langle ({B(x)^{-1}})^TC(x)a_e(x)C(x)^T B(x)^{-1}\eta,\eta\rangle=\,\langle a(x)\eta,\eta\rangle
\end{split}
\end{equation}
with 
\[
a(x):=\,(B(x)^{-1})^TC(x)a_e(x)C(x)^T B(x)^{-1}\,,
\]
$m\times m$ symmetric matrix. Then $f(x,\cdot)$ turns out to be a quadratic form on $\Rm$, induced by the matrix $a(x)$ for a.e.  $x\in\Om$.  Thus \eqref{finJ1} follows.

{\bf 2nd case. } Let us now deal with the subclass $J_2$. Let $(f_h)_h\subset J_2$. Notice that $f_h:\,\Rm\to [0,\infty)$, $h\in\bbN$, is a sequence of locally bounded, convex functions. Thus, by a well-known result  (see, for instance, \cite[Proposition 5.11]{DM}), we can infer that $(f_h)_h$ is also locally equi-Lipschitz continuous. From Ascoli-Arzel\`a'  s theorem, we can assume that, up to a subsequence, there exists $f\in J_2$ such that\begin{equation}\label{convpointfh}
f_h\rightarrow f\text{ uniformly on bounded sets of }\Rn\text{ as }h\to\infty\,.
\end{equation}
Let us define $\tilde F:\,L^p(\Om)\times\mA\to [0,\infty]$ as
\[
\tilde F(u,A):=
\displaystyle{\begin{cases}
\int_{A}f(Xu(x))dx&\text{ if }A\in\mA,u\in W^{1,p}_X(A)\\
\infty&\text{ otherwise}
\end{cases}
\,.
}
\]
Let us now prove that, for each $A\in\mA$,
\begin{equation}\label{convpointfunctls}
\lim_{h\to\infty}F^*_h(u,A)=\,\tilde F(u,A)\quad\forall\,u\in W^{1,p}_X(A)\,.
\end{equation}
Let us fix $A\in\mA$ and  $u\in W_X^{1,p}(A)$. Since $|Xu(x)|<\,\infty$ for a.e. $x\in A$, by \eqref{convpointfh}, it follows that
\begin{equation}\label{convpointintfunct}
\lim_{h\to\infty}f_h(Xu(x))=\,f(Xu(x))\text{ for a.e. }x\in A\,.
\end{equation}
On the other hand, as
\[
0\le\,f_h(Xu(x))\le\,c_1(1+|Xu(x)|^p)\text{ for a.e. }x\in A,\text{for each }h\,,
\]
by \eqref{convpointintfunct} and the dominated convergence theorem, \eqref{convpointfunctls} follows. We have only to prove that
\begin{equation}\label{FeqtildeF}
F(u,A)=\,\tilde F(u,A)\quad\forall\,A\in\mA,\,\forall\,u\in L^p(\Om)\,
\end{equation}
in order to get our desired conclusion. By \eqref{domGammalimit}, it is  sufficient to prove \eqref{FeqtildeF} for each $A\in\mA$ and for each $u\in W_X^{1,p}(A)$. 
The inequality
\begin{equation}\label{FletildeF}
F(u,A)\le\,\tilde F(u,A)\quad\forall\,A\in\mA,\,\forall\,u\in W_X^{1,p}(A)\,,
\end{equation}
follows by noticing that, for each $u\in W_X^{1,p}(A)$, by  $\Gamma-\liminf$ inequality and \eqref{convpointfunctls}
\[
F(u,A)\le\liminf_{h\to\infty}F^*_h(u,A)=\,\tilde F(u,A)\,.
\]
Let us now prove the opposite inequality
\begin{equation}\label{FgetildeF}
F(u,A)\ge\,\tilde F(u,A)\quad\forall\,A\in\mA,\,\forall\,u\in W_X^{1,p}(A)\,.
\end{equation}

Let us first recall that, for each $A\in\mA$, by \eqref{convpointfunctls} and  Proposition \ref{convpointimplgamma},
\begin{equation}\label{gammaconvtildeF}
\tilde F(u,A)=\,(
\Gamma(W^{1,p}_X(A))-\lim_{h\to\infty}F^*_h)(u)\quad\forall\,u\in W^{1,p}_X(A)\,.
\end{equation}
Fix $A\in\mA$ and let $u\in L^p(\Om)$ with $u|_A\in W^{1,p}_X(A)$. By the  $\Gamma-\lim$ equality, there exists a sequence $(u_h)_h\subset L^p(\Om)$ such that
\begin{equation}\label{convuhLp}
u_h\rightarrow u\text{ in }L^p(\Om)\text{, as }h\to\infty
\end{equation}
and 
\begin{equation}\label{gammalimsupcond}
\lim_{h\to\infty}F_h^*(u_h,A)=\,F(u,A)<\,\infty\,.
\end{equation}
By \eqref{gammalimsupcond}, we can assume that
\begin{equation}\label{uhinWX}
(u_h|_A)_h\subset W^{1,p}_X(A)\,.
\end{equation}
%For each $h$, let us still denote  $u_h$ the function defined on the whole $\Rn$ such that $u_h=\,0$ on $\Rn\setminus\Om$ and let
%\[
%u_{h,\eps}(x)=\,(\rho_\eps*u_h)(x)\text{ if }x\in\Rn\,,
%\]
%where $(\rho_\eps)_\eps$ is the family of mollifiers defined in \eqref{intrmoll} and $\rho_\eps*u_h$ is the convolution defined in \eqref{intrconvol}.  
Let $A'\in\mA$  with $A'\Subset A$. From Proposition \ref{convolconv} (ii),  if $w:=\,\overline{Xu_h}:\,\Rn\to\Rm$, that is, $\overline{Xu_h}=\,Xu_h$ on $A$ and $\overline{Xu_h}=\,0$ outside, for each $0<\,\eps<\,\dist(A',\Rn\setminus A)$
\begin{equation}\label{Jensentypeh}
\begin{split}
\int_{A'}f_h(\rho_\ep\ast \overline{Xu_h})\,dx\le\,\int_{A}f_h( \overline{Xu_h})\,dx\text{ for each }h\,.
\end{split}
\end{equation}
By \eqref{convuhLp}, \eqref{uhinWX} and Proposition \ref{convolconv} (i), for given $0<\,\eps<\,\dist(A',\Rn\setminus A)$,
\begin{equation}\label{convunifXrhoep}
X(\rho_\ep\ast \bar u_h)\rightarrow X(\rho_\ep\ast \bar u)\text{ uniformly on }A'\text{ as }h\to\infty
\end{equation}
and 
\begin{equation}\label{convunifrhoepX}
\rho_\ep\ast \overline{Xu_h}\rightarrow \rho_\ep\ast \overline{Xu}\text{ uniformly on }A'\text{ as }h\to\infty\,.
\end{equation}
In particular,
\begin{equation}\label{convconvol}
\rho_\eps*\bar u_h\rightarrow \rho_\eps*\bar u\text{ in } W^{1,p}_X(A')\text{ as }h\to\infty.
\end{equation}
Observe now that, by \eqref{Jensentypeh}, for each $0<\,\eps<\,\dist(A',\Rn\setminus A)$, for each $h$,
\begin{equation}\label{ineqFasth}
\begin{split}
&F^*_h(\rho_\ep\ast \bar u_h,A')=\,\int_{A'}f_h(X(\rho_\ep\ast \bar u_h))\,dx\\
&=\,\int_{A'}f_h(\rho_\ep\ast \overline{Xu_h})\,dx+\int_{A'}\left(f_h(X(\rho_\ep\ast \bar u_h))-f_h(\rho_\ep\ast \overline{Xu_h})\right)\,dx\\
&\le\,\int_{A}f_h( Xu_h)\,dx+\int_{A'}\left(f_h(X(\rho_\ep\ast \bar u_h))-f_h(\rho_\ep\ast \overline{Xu_h})\right)\,dx\\
&=\,F^*_h(u_h,A)+\,R_{\ep,h}.
\end{split}
\end{equation}
From \eqref{convpointfh}, \eqref{convunifXrhoep} and \eqref{convunifrhoepX}, it follows that, for given $0<\,\eps<\,\dist(A',\Rn\setminus A)$
\begin{equation}\label{limhReph}
\lim_{h\to\infty}R_{\ep,h}=\,R_\ep:=\,\int_{A'}\left(f(X(\rho_\ep\ast \bar u))-f(\rho_\ep\ast \overline{Xu})\right)\,dx\,.
\end{equation}

For given $0<\,\eps<\,\dist(A',\Rn\setminus A)$, by \eqref{gammaconvtildeF}, \eqref{gammalimsupcond}, \eqref{convconvol}, and \eqref{limhReph}, passing to the limit in \eqref{ineqFasth} as $h\to\infty$, it follows that
\begin{equation}\label{ineqtildeFFep}
\begin{split}
\tilde F(\rho_\eps*\bar u,A')&\le\,
\liminf_{h\to\infty}F_h^*(\rho_\eps*\bar u_h,A')\\
&\le\,\lim_{h\to\infty}F_h^*(u_h,A)+\lim_{h\to\infty} R_{\ep,h}=\,F(u,A)+R_\ep.
\end{split}
\end{equation}
Let us now show that
\begin{equation}\label{limepRep}
\lim_{\ep\to 0^+}R_\ep=\,0\,.
\end{equation}
Indeed
\[
X(\rho_\ep\ast \bar u)\rightarrow Xu\text{ and }\rho_\ep\ast \overline{Xu }\rightarrow Xu\text{ in }L^p(A')\text{, as } \ep\to 0^+
\]
and
\[
f(X(\rho_\ep\ast \bar u))\le\,c_1(1+|X(\rho_\ep\ast \bar u)|^p)\text{ and }f(\rho_\ep\ast \overline{Xu})\le\,c_1(1+|\rho_\ep\ast \overline{Xu}|^p)\text{ a.e. in }A'\,.
\]
Since $f$ is continuous, from Vitali's convergence theorem, \eqref{limepRep} follows.
By the semicontinuity of $\tilde F$, with respect to the $L^p$-topology, and by \eqref{limepRep}, we can pass to the limit as $\eps\to 0^+$ in \eqref{ineqtildeFFep} and we get
\begin{equation}\label{ineqtildeFF}
\tilde F(u,A')\le\,\lim_{\eps\to 0^+}\tilde F(\rho_\eps*\bar u,A')\le\,F(u,A)\text{ for each }A'\Subset A\,.
\end{equation}
Finally, taking the supremum in  \eqref{ineqtildeFF} on all $A'\in\mA$ with $A'\Subset A$, we get \eqref{FgetildeF}.
\end{proof}


\begin{thebibliography}{100}



%%biblioHMJ
%\ref
%\key A
%\by R. A. Adams
%\book Sobolev Spaces
%\publ Academic Press
%\yr 1975
%\publaddr New York
%\endref
%
%
%\ref
%\key AB
%\by M. Amar, G. Bellettini
%\paper A notion of total variation depending on a metric 
%with discontinuous coefficients
%\jour Ann. Inst. H. Poincar\'e, An. non lin\'eaire
%\vol 11 
%\yr 1994
%\pages 91--133
%\endref
%
%\ref
%\key A
%\by G. Anzellotti
%\paper The Euler equation for functional with linear grothw 
%\jour Trans. Amer. Math. Soc.
%\vol 290
%\yr 1985
%\pages 483--501
%\endref
%
%\ref
%\key AG
%\by G. Anzellotti, M. Giaquinta
%\paper Funzioni $BV$ e tracce
%\jour Rend. Sem. Mat. Univ. Padova
%\vol 60
%\yr 1978
%\pages 1--22
%\endref
%
\bibitem[ACM]{ACM}{\sc F. Acanfora,G. Cardone, S. Mortola}, {\em On the variational convergence of non-equicoercive quadratic integral functionals and semicontinuity problems}, NoDEA {\bf 10} (2003), 347--373.

\bibitem[APS]{APS}{\sc L. Ambrosio, A. Pinamonti, G. Speight},{\em Weighted Sobolev Spaces on Metric Measure Spaces}, Journal fur die Reine und Angewandte Mathematik (Crelle's Journal) {\bf 746} (2019), 39--65.


%\bibitem[AB]{AB}{\sc C.D. Aliprantis, K.C. Border}, {\em Infinite Dimensional Analysis}, 3rd Ed., Springer, 2006.

\bibitem[BFTT]{BFTT} {\sc A. Baldi, B. Franchi, N. Tchou, M.C. Tesi}, {\em Compensated compactness for differential forms in Carnot groups and applications}, Adv. Math.  {\bf 223} (2010),  1555--1607.

\bibitem[BFT]{BFT} {\sc A. Baldi, B. Franchi, M.C. Tesi},  {\em Compensated compactness, div-curl theorem and H-convergence in general Heisenberg groups}, Lect. Notes Seminario Interdisciplinare di Matematiche {\bf 6} (2007), 33--47.






%\bibitem[BBP]{BBP}{\sc M. Bramanti, L. Brandolini, M. Pedroni}, {\em Basic properties of nonsmooth H\"ormander's vector fields and Poincar\'e's inequality} Forum Mathematicum  {\bf 25},(4) (2013), 703--769.
%\ref
%\key BiMo
%\by  Biroli, U. Mosco
%\paper  Sobolev and isoperimetric inequalities for Dirichlet forms
%on homogeneus space
%\jour to appear
%\endref
\bibitem [BMT]{BMT} {\sc M. Biroli, U. Mosco, N. Tchou},  {\em Homogenization for degenerate operators with periodic coefficients with respect to the Heisenberg group}, C. R. Acad. Sci. Paris, {\bf 322} (1996)., Série I 439--444.

\bibitem[BPT1]{BPT1} {\sc M. Biroli, C. Picard, N. Tchou},  {\em Homogenization of the p-Laplacian associated with the Heisenberg group},  Rend. Accad. Naz. Sci. XL Mem. Mat. Appl. (5) {\bf 22} (1998), 23--42.


\bibitem[BPT2]{BPT2} {\sc M. Biroli, C. Picard,  N. Tchou},  {\em Asymptotic behavior of some nonlinear subelliptic relaxed Dirichlet problems} Rend. Accad. Naz. Sci. XL Mem. Mat. Appl. (5) {\bf 26} (2002), 55--113.

\bibitem[BT]{BT} {\sc M. Biroli,  N. Tchou},  {\em $\Gamma$-convergence for strongly local Dirichlet forms in perforated domains with homogeneous Neumann boundary conditions}, Adv. Math. Sci. Appl. {\bf 17} (2007), no. 1, 149--179.


%\bibitem[BMV]{BMV} {\sc M. Biroli, P,G.Vernole},  {\em Strongly local nonlinear Dirichlet functionals and forms}, Adv. Math. Sci. Appl. {\bf 15} (2005), no. 2, 655--682.
%
\bibitem[BLU]{BLU}{\sc A. Bonfiglioli, E. Lanconelli, F. Uguzzoni}, {\em Stratified Lie Groups and Potential Theory for Their Sub-Laplacians}, Springer, 2007.
%
%\ref
%\key Bo
%\by J.-M.~Bony
%\paper Principe de maximum, in\'egalit\'e de Harnack et
%unicit\'e du probl\`eme de Cauchy pour les op\'erateurs
%elliptiques d\'eg\'en\'er\'es
%\jour Ann. Inst. Fourier, Grenoble
%\vol 19
%\yr 1969
%\pages 277--304
%\endref
%
\bibitem[Bra]{Bra}{\sc A. Braides}, {\em A handbook of $\Gamma$-convergence}, Handbook of Differential Equations.Stationary Partial Differential Equations {\bf 3}, M. Chipot and P. Quittner, eds., Elsevier, 2006.
\bibitem[B]{B}{\sc G. Buttazzo}, {\em Semicontinuity, relaxation and integral rapresentation}, Longman, Harlow, 1989.
%\ref 
%\key   B  
%\by G. Buttazzo
%\book Semicontinuity, relaxation and integral rapresentation
%in the calculus of variations
%\publ Longman
%\yr 1989
%\publaddr Harlow
%\endref
%
%\bibitem[BM]{BM} {\sc G. Buttazzo, V.G. Mizel}, {\em Interpretation of Lavrentiev phenomenon by relaxation}, J. Funct. Anal., {\bf 110} (1992), 409--424.

%\ref 
%\key   BM  
%\by G. Buttazzo, V.G. Mizel
%\paper Interpretation of Lavrentiev phenomenon by relaxation
%\jour J. Funct. Anal.
%\vol 110
%\yr 1992
%\pages 409--424
%\endref
%
%\ref
%\key CDG
%\by L. Capogna, D. Danielli, N. Garofalo
%\paper The geometric Sobolev embedding for vector fields
%and the isoperimetric inequality
%\jour Comm. in Analysis and Geometry
%\vol 2
%\yr 1994
%\pages 203--215
%\endref
%
%\bibitem[CM]{CM}{\sc E. Comi, V. Magnani}, {\em The Gauss-Green Theorem in Stratified Groups}, arXiv:1806.04011v2.
%
%
%\ref
%\key DM
%\by G. Dal Maso
%\paper Integral rapresentation on $BV(\O)$ of $\Gamma$-limits of
%variational integrals
%\jour  Manuscripta Math.
%\vol 30
%\yr 1980
%\pages 387--416
%\endref
%
%
\bibitem[DM]{DM}{\sc G. Dal Maso}, {\em An Introduction to $\Gamma$-convergence}, Birhk\"auser, Boston, 1993.
%\ref
%\key D
%\by S. Delladio
%\paper Lower semicontinuity and continuity of function measures with
%respect to the strict convergence
%\jour Proc. Royal Soc. of Edinburgh
%\vol 119A
%\yr 1991
%\pages 265--278
%\endref
%
%\ref
%\key DG1
%\by E. De~Giorgi
%\paper Conversazioni di matematica
%\jour Quaderni Dip. Mat. Univ. Lecce (1988-1990)
%\vol 2
%\yr 1990
%\endref
%
%\ref
%\key DG2
%\bysame
%\paper Alcuni problemi variazionali della geometria
%\book Conferenze in onore di G. Acquaro, Bari
%\publ Conferenze del Seminario Matematico 
%\publaddr Bari
%\yr 1991
%\pages 112--125
%\endref
%
%\ref
%\key DG3
%\bysame
%\paper Su alcuni problemi comuni all'analisi ed alla geometria
%\jour Note di Matematica
%\vol IX-Suppl.
%\yr 1989
%\pages 59--71
%\endref
%
%\ref
\bibitem[ET]{ET}{\sc I. Ekeland, R. T\'eman}, {\em Convex Analysis and Variational Problems}, Classic in Applied Mathematics {\bf 28}, SIAM , Philadelphia, 1999.
%\key EG
%\by L. Evans, R. Gariepy
%\book Lecture Notes on Measure Theory and Fine Properties
%     of Functions 
%\publ CRC Press
%\yr 1992
%\publaddr Boca Raton
%\endref
%
%\ref 
%\key F
%\by K.O. Friedrichs
%\paper The identity of weak and strong estension of differential
%operators
%\jour Trans. Amer. Math. Soc.
%\vol 55
%\yr 1944
%\pages 132--151
%\endref
%
%
%
%\ref 
%\key FP
%\by C.~Fefferman, D.H.~Phong
%\paper  Subelliptic eigenvalue problems
%\inbook Conference on Harmonic Analysis, Chi\-ca\-go, 1980, W. Beckner et 
%al. ed.
%\publ Wadsworth 
%\yr 1981
%\pages 590--606
%\endref
%
%\ref 

\bibitem[FS]{FS}{\sc G.B. Folland, E.M. Stein,}, {\em Hardy spaces on homogeneous groups}, Princeton University
Press, Princeton,1982.
%\key FGW1
%\by  B.~Franchi, S.~Gallot, R.L.~Wheeden 
%\paper Sobolev
%and isoperimetric inequalities for degenerate metrics
%\jour Math. Ann.
%\yr 1994
%\vol 300
%\pages 557--571
%\endref
%
%\ref 
%\key FGW2
%\bysame 
%\paper In\'egalit\'es isop\'erimetriques pour des m\'etriques
%d\'e\-g\'e\-n\'e\-r\'ees
%\jour C.~R. Acad. Sci. Paris, Ser. A  
%\vol 317 
%\yr  1993
%\endref 

\bibitem[FGVN]{FGVN}{\sc B. Franchi, C.E. Guiti\'errez, T. van Nguyen}, {\em Homogeneization and Convergence of Correctors in Carnot Groups}, Comm. PDE {\bf 30} (2005), 1817--1841.
%
%
%
%
%\bibitem[FL]{FL} {\sc B. Franchi, E. Lanconelli}, {\em H\"older regularity theorem for a class of linear nonuniformly elliptic 
%operators with measurable coefficients}, Ann. Scuola Norm. Sup. Pisa, 10{\bf 4} (1983), 523--541.

%\bibitem[F1]{F1}{\sc B. Franchi}, {\em Weighted Sobolev-Poincar\'e Inequalities and Pointwise Estimates for a class of degenerate elliptic equations}, Trans. Amer. Math. Soc. {\bf 327} (1991), 125--158.

%\bibitem[F2]{F2}{\sc B. Franchi}, {\em In\'egalit\'es de Sobolev pour des champs de vecteurs lipschitziens}, C.R. Acad. Sci. Paris - Ser. A {\bf 311} (1990), 329--332.
%\ref 
%\key FL
%\by B.~Franchi, E.~Lanconelli
%\paper  H\"older regularity theorem for a class of linear nonuniformly elliptic 
%operators with measurable coefficients
%\jour Ann. Scuola Norm. Sup. Pisa
%\vol 10 (4)  
%\yr 1983
%\pages 523--541
%\endref
%
%
%\bibitem[FLW]{FLW} {\sc B. Franchi, G. Lu, R. L. Wheeden}, {\em Representation formulas and weighted Poincar\'e inequalities for H\"ormander vector fields}, Ann. Inst. Fourier, Grenoble 45{\bf 2} (1995), 577--604.

%\ref
%\key FLW
%\by B. Franchi, G. Lu, R.L. Wheeden
%\paper Representation formulas and weighted 
%Poincar\'e inequalities for H\"ormander vector fields
%\jour Ann. Inst. Fourier, Grenoble
%\vol 45 (2) 
%\yr 1995
%\pages 577--604
%\endref
%

%\bibitem[FS]{FS} {\sc G. Folland, E. Stein}, {\em Hardy spaces on homogeneous groups}. Mathematical Notes, 28. Princeton University Press, N.J.; University of Tokyo Press, Tokyo, 1982.
\bibitem[FSSC1]{FSSC1} {\sc B.~Franchi, R.~Serapioni, F.~Serra Cassano}, {\em Meyers-Serrin type theorems and relaxation of variational integrals depending on vector fields}, Houston J. Math., {\bf 22} (1996), 859--889.

%\ref
%\key FSSC1
%\by B.~Franchi, R.~Serapioni, F.~Serra Cassano
%\paper Lavrentiev phenomenon and functionals which are
%coercitive with respect to vector fields
%\jour  C.R. de l'Acad. Sc., Paris Serie I 
%\vol 320 
%\yr 1995
%\pages 695-698
%\endref
%
\bibitem[FSSC2]{FSSC2} {\sc B.~Franchi, R.~Serapioni, F.~Serra Cassano}, {\em Approximation and Imbedding Theorems for Weighted Sobolev Spaces Associated with Lipschitz Continuous Vector Fields}, Bolletino U.M.I, {\bf 11-B} (1997), 83--117.

%\ref
%\key FSSC2
%\bysame
%\paper Approximation and imbedding theorems for weighted Sobolev
%spaces associated with Lipschitz continuous vector fields
%\jour Boll. Un. Mat. Ital.
%\toappear
%\endref
%
\bibitem[FTT]{FTT} {\sc B. Franchi, N. Tchou, M.C. Tesi},  {\em Div- Curl Type Theorem, H-convergence and Stokes Formula in the Heisenberg Group}, Comm. Contemp. Math.  {\bf 81}, no. 1 (2006), 67--99.

\bibitem[FT]{FT} {\sc B. Franchi,  M.C. Tesi}, {\em Two-scale convergence in the Heisenberg group}, J. Math. Pures Appl. {\bf 81} (2002), 495--532.

\bibitem[Fu]{Fu}{\sc M. Fukushima}, {\em Dirichlet Forms and Symmetric Markov Processes}, North-Holland, Amsterdam, 1980.

%\ref 
%\key   G  
%\by E. Giusti
%\book Minimal Surfaces and Functions of Bounded Variation
%\publ Birkh\"auser
%\yr 1985
%\publaddr Boston
%\endref
%

\bibitem[GN]{GN} {\sc N.Garofalo, D.M Nhieu}, {\em Isoperimetric  and Sobolev Inequalities for Carnot-Carath\'eodory spaces and the existence of minimal surfaces}, Comm. Pure Appl. Math., {\bf 49} (1996), 1081--1144.

%\ref
%\key GN1
%\by N. Garofalo and D.M. Nhieu
%\paper Isoperimetric and Sobolev inequalities 
%for Carnot-Carath\'eodory spaces and the existence of minimal surfaces
%\jour Comm. Pure Appl. Math.
%\toappear
%\endref
%
%\ref
%\key GN2
%\bysame
%\paper A general extension theorem for Sobolev functions on
%Carnot-Carath\'eodory spaces
%\yr 1995
%\finalinfo preprint
%\endref
%
%
%
%\ref
%\key GS
%\by C. Goffman, J. Serrin
%\paper Sublinear functions of measures and variational integrals
%\jour Duke Math. J.
%\vol 31
%\yr 1964
%\pages 159--178
%\endref
%
%
%\ref
%\key Gr
%\by M. Gromov
%\paper Carnot-Caratheodory spaces seen from within
%\jour preprint of the IHES
%\yr 1994
%\endref
%
%\ref
%\key H
%\by L.~H\"ormander
%\paper Hypoelliptic second order differential equations
%\jour Acta Math.
%\yr 1967
%\vol 119
%\pages 147--171
%\endref
%
%

%\bibitem[HK]{HK}{\sc P.~Hajlasz, P. Koskela}, {\em Sobolev Met Poincar\'e}. Memoirs of the AMS, {\bf 145}(688) (2000).

%\bibitem[Jer]{J}{\sc D.~Jerison}, {\em The Poincar\'e inequality for vector fields satisfying H\"ormanderâs condition.}
%Duke Math. J. 53 (1986), no. 2, 503--523.

%\bibitem[L]{L} {\sc M.~Lavrentiev}, {\em Sur quelques probl\`emes du calcul des variations}, Ann. Mat. Pura Appl., {\bf 4} (1926), 107--124.
\bibitem[La]{La}{\sc P. D. Lax}, {\em Linear Algebra and its Applications}- 2nd Ed., Wiley, Hoboken, NewJersey, 2007.
%\ref
%\key L
%\by M.~Lavrentiev
%\paper Sur quelques probl\`emes du calcul des variations
%\jour Ann. Mat. Pura Appl.
%\vol 4
%\yr 1926
%\pages 107--124
%\endref
%
%\ref
%\key LM
%\by S.~Luckhaus, L.~Modica
%\paper The Gibbs Thompson relation within the gradient theory of 
%phase transition
%\jour Arch. Rat. Mech. Anal.
%\vol 107
%\yr 1989
%\pages 71--83
%\endref
%
\bibitem[MR]{MR}{\sc Z.-M. Ma, M. R\"ockner}, {\em Introduction to the Theory of (Non-symmetric) Dirichlet forms}, Universitext, Springer-Verlag Berlin Heidelberg 1992.
%\bibitem[Ma]{Ma}{\sc V. Magnani}, {\em Towards differential calculus in stratified groups}, J. Aust. Math. Soc. {\bf 95} (2010), 76--128.
\bibitem[MPSC]{MPSC}{\sc A. Maione, A. Pinamonti, F. Serra Cassano}, {\em $\Gamma$-convergence  for functionals depending on vector fields. II. Convergence of minimizers.}, forthcoming.
%\bibitem[MPSC2]{MPSC2}{\sc A. Maione, A. Pinamonti, F. Serra Cassano} forthcoming.
\bibitem[MS]{MS} {\sc N.G. Meyers, J. Serrin}, {\em H = W}, Proc. Nat. Acad. Sci. USA, {\bf 51} (1964), 1055--1056.

%\bibitem[NSW]{NSW}{\sc A. Nagel, E. M. Stein, S. Wainger}, {\em Balls and metrics defined by vector fields I: basic properties}, Acta Math. {\bf 155} (1985), 103-147.

%\ref
%\key MS
%\by N.G. Meyers, J. Serrin
%\paper H = W
%\jour Proc. Nat. Acad. Sci. USA
%\vol 51
%\yr 1964
%\pages 1055--1056
%\endref
\bibitem[Mo]{Mo}{\sc U. Mosco}, {\em Composite Media and Asymptotic Dirichlet Forms}, J. Funct. Anal. , {\bf 123} (1993) , 368--421.
%
%\ref
%\key R
%\by Yu. G. Reschetnyak
%\paper Weak convergence of completely additive vector functions on a set
%\jour Sibirsk. Mat. Z.
%\vol 9
%\yr 1968
%\pages 1386--1394
%\finalinfo (translated)
%\endref
\bibitem[Ro]{Ro}{\sc R.T. Rockafellar}, {\em Convex Analysis}, Princeton University Press, Princeton, 1972.
%
%
%\ref
%\key S
%\by J. Serrin
%\paper On the definition and properties of certain variational
%integrals
%\jour Trans. Amer. Math. Soc.
%\vol 101
%\yr 1961
%\pages 139--167
%\endref
%
%
%\ref 
%\key   SV  
%\by D.W. Stroock, S.R.S. Varadhan
%\book Multidimensional Diffusion Processes
%\publ Springer Verlag
%\yr 1979
%\publaddr Berlin, New York
%\endref
%
%\ref
%\key Su
%\by H. Sussmann
%\paper Orbits of families of vector fields and integrability
%of distributions
%\jour Trans. Amer. Math. Soc.
%\vol 180
%\yr 1973
%\pages 171--188
%

%%finebiblioHMJ


%\bibitem{A} {\sc L.Ambrosio}, {\em Some fine properties of sets of finite perimeter in Ahlfors regular metric measure spaces}, Adv. in Math., {\bf 159} (2001), 51--67.
%
%\bibitem{ambrkir} {\sc L.Ambrosio, B.Kirchheim}, {\em Rectifiable sets in metric and Banach spaces}, Math. Ann., {\bf 318} (2000), 527--555.
%
%\bibitem{ambrkir2} {\sc L.Ambrosio, B.Kirchheim}, {\em Currents in metric spaces}, Acta Math., {\bf 185} (2000), 1--80.
%
%
%\bibitem{ASCV} {\sc L.Ambrosio, F. Serra Cassano, D. Vittone},  {\em Intrinsic Regular Hypersurfaces in Heisenberg Group}, J.Geom. Anal. {\bf 2} (2006), 187--232.
%
%\bibitem{AS} {\sc G.Arena, R.Serapioni}, {\em Intrinsic regular Submanifolds
%in Heisenberg Groups are differentiable graphs}, Calc.Var. and PDE, {\bf 4} (2009), 517--536.
%
%\bibitem{BSC} {\sc F. Bigolin, F. Serra Cassano}, {\em Intrinsic regular graphs in Heisenberg groups vs. weak solutions of non-linear first-order PDEs}, Adv. Calc. Var., {\bf 3} (2010), no. 1, 69--97.
%
%\bibitem{bigSC} {\sc F. Bigolin, F: Serra Cassano}, {\em Distributional solutions of Burgers' equation and intrinsic regular graphs in Heisenberg groups}, J. Math. Anal. Appl, {\bf 362} (2010), no.2, 561--568.
%
%
%%\bibitem{BM} M. Biroli, U. Mosco, \textit{ Sobolev inequalities on homogeneous spaces}, Potential Anal. $4$ ($1995$), $311-324$.
%
%\bibitem{BLU} {\sc A.Bonfiglioli, E.Lanconelli, F. Uguzzoni}, {\em Stratified Lie groups and potential theory for their sub-laplacians}, Springer-verlag, (2007).
%
%%\bibitem{BBP} M. Bramanti, L. Brandolini, M. Pedroni, \textit{Basic properties of nonsmooth HÃ¯Â¿Årmander's vector fields and PoincarÃ¯Â¿Å's inequality},  preprint.
%
%%\bibitem{CDG} L. Capogna, D. Danielli, N. Garofalo, \textit{ Subelliptic mollifiers and a basic pointwise estimate of Poincar\'e type},  Math. Z. ($1997$).
%%\bibitem{CCM2} {\sc L.Capogna, G.Citti, M.Manfredini}, {\em Regularity of non-characteristic minimal graphs in Heisenberg groups $\bbH^1$}, Indiana Univ.Math.J., {\bf 58} (2009), no. 5, 2115--2160.
%
%%\bibitem{CCM3} {\sc L.Capogna, G.Citti, M.Manfredini}, {\em Regularity of non-characteristic minimal graphs in Heisenberg groups $\bbH^n$, $n>1$}, J. Reine Angew. Math. {\bf 648} (2010), 75Ã¯Â¿Å110.
%
%\bibitem{CDGisop} {\sc L.Capogna, D.Danielli, N.Garofalo}, {\em The Geometric Sobolev embedding for vector fields and the isoperimetric inequality}, Comm.Anal.Geom., {\bf 2} (1998), no.4-5, 403--432.
%
%\bibitem{Cap} {\sc L.Capogna, D.Danielli, S.D.Pauls, J.T.Tyson}, {\em An introduction to the Heisenberg Group and the Sub-Riemannian Isoperimetric Problem}, PM $259$, Birkh$\ddot{\mbox{a}}$user, (2007).
%
%\bibitem{CK} {\sc J.  Cheeger, B. Kleiner}, {\em Differentiating maps into $L^1$, and the geometry of BV functions}, Ann. of Math.(2), {\bf 171} (2010), no. 2, 1347--1385.
%
%%\bibitem{CGL} G.Citti, N. Garofalo, E. Lanconelli, \textit{Harnack's inequality for sum of
%%squares of vector fields plus a potential}, Amer. J. Math., 115,
%%1993, 699-734.
%
%%\bibitem{CLM} G. Citti, E. Lanconelli, A. Montanari, { \it Smoothness of Lipschitz continuous graphs with non vanishing Levi curvature}, Acta Math. 188 (2002), no. 1, 87-128.
%
%\bibitem{CM2} {\sc G. Citti, M. Manfredini}, {\em Blow-up in non homogeneous Lie groups and rectifiability}, Houston J. Math. {\bf 31} (2005), no. 2, 333--353.
%
%\bibitem{CM} {\sc G. Citti, M. Manfredini}, {\em Implicit Function Theorem in Carnot-Carath\'eodory spaces}, Comm. Contemp. Math., {\bf 5} (2006), 657--680.
%
%\bibitem{CMS} {\sc G. Citti, M. Manfredini, A. Sarti}, {\em Neuronal oscillations in the visual cortex: $\Gamma$-convergence to the Riemannian Mumford-Shah functional}, SIAM J. Math. Anal. {\bf 35} (2004), no. 6, 1394--1419.
%
%\bibitem{CS} {\sc G. Citti, A. Sarti}, {\em A cortical based model of perceptual completion in the roto-translation space}, J. Math. Imaging Vision {\bf 24} (2006), no. 3, 307--326.
%
%\bibitem{DGCP} {\sc E.De Giorgi, F.Colombini , L.C.Piccinini}, {\em Frontiere orientate di misura minima e questione collegate}, Pubblicazione della classe di Scienze della Scuola Normale Superiore, Pisa, 1972.
%
%\bibitem{FS} {\sc G.B. Folland, E.M.Stein}, {\em Hardy spaces on homogeneous groups}, Princeton University Press, 1982.
%
%%\bibitem{FGW} B. Franchi, C. Guti\'errez, R. L. Wheeden, \textit{ Weighted Sobolev-Poincar\'e inequalities for
%%Grushin type operators}, Comm. Partial Differential Equations, $19$ ($1994$), no. $3-4$, $523Ã¯Â¿Å604$.  %\red{ non in rete}
%
%%\bibitem{FL2} B. Franchi, E. Lanconelli,  \textit{ H\"older regularity theorem for a class of linear nonuniformly elliptic operators with measurable coefficients}, Ann. Scuola Norm. Sup. Pisa Cl. Sci., ($4$),
%%$10$ ($1983$), $523-541$.
%
%
%
%%\bibitem{FLW}  B. Franchi, G. Lu, R.L. Wheeden,  \textit{ Representation formulas and weighted Poincar\'e
%%inequalities for H\"omander vector fields} Ann. Inst. Fourier (Grenoble) $45$ ($1995$), $577Ã¯Â¿Å604$.
% %\red{ formula media $|f-f_B|\le ...$ implica poincare  per campi di Hormander, migliora la formula di rappr di Lu}
%
%%\bibitem{FLW2} B. Franchi, G.  Lu, R. Wheeden,   \textit{ A relationship between Poincar\'e type inequalities and representation formulas in spaces of homogeneous type}, Internat. Math. Res. Notices, no. $1$, ($1996$),$1-14$. %\red{poincare implica formula media}
%
%
%
%\bibitem{FSSC2} {\sc B.Franchi, R.Serapioni, F.Serra Cassano}, {\em Rectifiability and Perimeter in the Heisenberg group}, Math.Ann., {\bf 321} (2001), 479--531.
%
%\bibitem{FSSC3} {\sc B.Franchi, R.Serapioni, F.Serra Cassano}, {\em On the structure of Finite Perimeter Sets in Step $2$ Carnot Groups}, J.Geom.Anal., {\bf 13} (2003), 421--466.
%
%\bibitem{FSSC4} {\sc B.Franchi, R.Serapioni, F.Serra Cassano}, {\em
%Regular submanifolds, graphs and area formula in Heisenberg groups}, Adv. in Math. {\bf 211} (2007), 157--203.
%
%\bibitem{FSSC}{\sc B.Franchi, R.Serapioni, F.Serra Cassano}, {\em Differentiability of intrinsic Lipschitz Functions within Heisenberg Groups}, J.Geom. Anal.,DOI: 10.1007/s12220-010-9178-4.
%
%\bibitem{GAR} {\sc N.Garofalo, D.M Nhieu}, {\em Isoperimetric  and Sobolev Inequalities for Carnot-Carath\'eodory spaces and the existence of minimal surfaces}, Comm. Pure Appl. Math., {\bf 49} (1996), 1081--1144.
%
%\bibitem{GAN}{\sc N.Garofalo, D.Nhieu} {\em Lipschitz continuity, global smooth approximation and extension theorems for Sobolev functions in Carnot-Carath\'eodory spaces}, J. D'analyse Math\'ematique, {\bf 74}  (1998), 67--97.
%
%%\bibitem{GT} N.Gilbarg, N.S.Trudinger : \textit{Elliptic PDE of second
%%order} Springer Verlag
%
%\bibitem{Gro} {\sc M.Gromov}, {\em Carnot-Carath\'eodory spaces seen fromwithin}, in {\em Subriemannian Geometry}, Progress in Mathematics,{\bf 144}. ed. by A.Bellaiche and J.Risler, Birkhauser Verlag,Basel (1996)
%
%%\bibitem{HK} P. Hajlasz, P. Koskela, \textit{ Sobolev met Poincar\'e}, Mem. Amer. Math. Soc., $688$,  ($2000$).
%
%\bibitem{HP} {\sc R.K. Hladky,  S.D. Pauls}, {\em Minimal surfaces in the roto-translation group with applications to a neuro-biological image completion model}, J. Math. Imaging Vision 36 (2010), no. 1, 1--27.
%
%%\bibitem{H} L.H$\ddot{\mbox{o}}$rmander, \textit{Hypoelliptic second-order differential equations}, Acta Math. $121$, $147-171$, ($1968$)
%
%%\bibitem{J} D. Jerison, \textit{ The Poincar\'e inequality for vector fields satisfying H\"ormander's condition},
%%Duke Math. J.,  $53$, ($1986$), $503-523$.
%\bibitem{KSC} {\sc B.Kirchheim, F.Serra Cassano}, {\em Rectifiability and parametrization of intrinsic regular surfaces in the Heisenberg group}, Ann. Scuola Norm. Sup.  Pisa Cl. Sci. (5) {\bf III} (2004),871--896.
%
%
%%\bibitem{LM} E. Lanconelli, M. Morbidelli, \textit{  On the Poincar\'e inequality for vector fields}, Ark.
%%Mat.,  $38$,  ($2000$), $327-342$.
%
%
%
%%\bibitem{Lu} G. Lu,   \textit{ Weighted Poincar\'e and Sobolev inequalities for vector fields satisfying H\"ormanderÃ¯Â¿Ås condition and applications}, Rev. Mat. Iberoamericana, $8$ ($1992$), no. $3$,
%%$367Ã¯Â¿Å439$.  %\red{usa una formula di media con a secondo membro la funzione massimale del gradiente liftato (+la funzione) che non e' $L^1$ per p=1}
%
%%\bibitem{Lu1} G. Lu,  \textit{ The sharp Poincar\'e  inequality for free vector fields
%%Revista matemÃ¯Â¿Åtica iberoamericana}, $1994$  %\red{ poincare per campi, senza peso ed esponenti diversi Sobolev-Poincare}
%
%
%
%%\bibitem{LW} G.Lu, R. Wheeden, \textit{ An optimal representation formula for Carnot-Carath\'eodory vector fields}, Bull. London Math. Soc., $30$ ($1998$), $578-584$.   %\red{poincare implica formula media}
%
%%\bibitem{MSC} P. Maheux, L. Saloff-Coste,  \textit{ Analyse sur le boules d'un op\'erateur sous-elliptique},
%%Math. Ann.,  $303$,  ($1995$), $713-746$.
%
%\bibitem{magntesi} {\sc V.Magnani}, {\em Elements of Geometric Measure Theory on Sub- Riemannian Groups}, Tesi di Perfezionamento, Scuola Normale Superiore, Pisa, 2003.
%
%\bibitem{M2} {\sc V. Magnani}, {\em Characteristic points, rectifiability and perimeter measure on stratified groups}, J. Eur. Math. Soc. {\bf 8} (2006), no. 4, 585--609.
%
%
%\bibitem{M} {\sc V. Magnani}, {\em Towards differential calculus in stratified groups}, Preprint ($2007$).
%
%\bibitem{MAT} {\sc P. Mattila}, {\em Geometry of Sets and Measures in Euclidean Spaces}, Cambridge University Press, 1995.
%
%\bibitem{MSSC} {\sc P. Mattila, R. Serapioni, F. Serra Cassano}, {\em Characterizations of intrinsic rectifiability in Heisenberg groups}, Ann. Scuola Norm. Sup.  Pisa Cl. Sci. (5), {\bf 9} (2010), no.4, 687--723.
%
%%\bibitem{MM1} A. Montanari, D. Morbidelli,  Sobolev and Morrey estimates for non-smooth vector fields of step two, Z. Anal. Anwendungen 21 (2002), no. 1, 135--157
%
%%\bibitem{MM2} A. Montanari, D. Morbidelli, \textit{ Balls defined by nonsmooth vector fields and the Poincar\'e inequality}, Ann. Inst. Fourier, $54$  no. $2$, ($2004$), $431-452$.
%
%%\bibitem{MSCV} {\sc R.Monti, F.Serra Cassano, D.Vittone}, {\em A negative Answer to the Bernstein Problem for Intrinsic Graphs in the Heisenberg Group}, Bollettino U.M.I, {\bf 1} (2008), no.3, 709--727.
%
%\bibitem{MV} {\sc R.Monti, D.Vittone}, {\em Sets with finite $\bbH$-perimeter and controlled normal}, Math.Z., DOI: 10.1007/s00209-010-0801-7.
%
%
%%\bibitem{Morb} D. Morbidelli, \textit{ Fractional Sobolev norms and structure of the Carnot-Carath\'eodory balls
%%for H\"ormander vector fields}, Studia Math. , $139$,  ($2000$), $213-244$.
%%
%%\bibitem{NSW}  A. Nagel, E.M. Stein, S. Wainger, {\it Balls and metrics
%%defined by vector fields I: Basic properties}, Acta Math. $155$,
%%(1985), 103-147.
%
%\bibitem{Pansu} {\sc P.Pansu}, {\em M\'etriques de Carnot-Carath\'eodory quasiisom\'etries des espaces sym\'etriques de rang un}, Ann.Math., {\bf 129} (1989), 1--60.
%
%
%%\bibitem{RS1} F. Rampazzo, H. J. Sussmann, \textit{ Commutators of flow maps of nonsmooth vector fields}, J. Differential Equations 232 ($2007$), no. 1, 134-175. %ancora da leggere
%
%%\bibitem{RS} L. Rothschild, E.M. Stein, {\it
%%Hypoelliptic differential operators and nihilpotent Lie groups},
%%Acta Math.   $137$, ($1977$), $247-320$.
%
%\bibitem{pauls} {\sc S.D.Pauls}, {\em A notion of rectifiability modelled on Carnot groups}, Indiana Univ. Math. J. {\bf 53} (2004), 49--81.
%
%
%%\bibitem{SC}  L. Saloff-Coste, \textit{ A note on Poincar\'e, Sobolev and Harnack inequalities}, Internat. Math.
%%Res. Notices, no. $2$, ($1992$),  $27-38$.  %\red{ non in rete}
%
%\bibitem{SCP} {\sc A. Sarti, G. Citti, J. Petitot}, {\em The symplectic structure of the primary visual cortex}, Biol. Cybernet. {\bf 98} (2008), no. 1, 33--48.
%
%
%\bibitem{SCV} {\sc F.Serra Cassano, D.Vittone}, {\em Graphs of bounded variation, existence and local Boundedness of non-parametric minimal surfaces in the Heisenberg group}, url: http://cvgmt.sns.it/papers/servit/BVHgraphs101022.pdf.
%
%\bibitem{V} {\sc A.Visintin}, {\em Strong Convergence results related to strict convexity}, Comm. in PDE, {\bf 9} (1984), 439--466.
%\bibitem{vitttesi} {\sc D. Vittone}, {\em Submanifolds in Carnot groups}, Tesi di Perfezionamento, Scuola Normale Superiore, Pisa, 2008.
%
%\bibitem{vitt} {\sc D. Vittone}, {\em Lipschitz surfaces, perimeter and trace theorems for BV functions in Carnot-Carath\'eodory spaces}, to appear on Ann. Scuola Norm. Sup.
%



\end{thebibliography}
\end{document}